\documentclass[12pt,reqno,a4paper]{amsart}

\usepackage{amsmath,amssymb,amsthm,amsaddr}
\usepackage{enumitem}
\usepackage{mathtools}
\mathtoolsset{showonlyrefs=true}

\usepackage[utf8]{inputenc}
\usepackage[dvipsnames]{xcolor}

\definecolor{lightgray}{gray}{0.95}
\usepackage{tikz}

\usepackage[margin=3cm,top=3.5cm,bottom=3cm,footskip=1cm,headsep=1.5cm]{geometry}

\usepackage{hyperref}
\hypersetup{
    colorlinks,
    citecolor=red,
    filecolor=black,
    linkcolor=blue,
    urlcolor=black
}

\usepackage{mathtools}
\usepackage{tikz}
\usetikzlibrary{calc}
\usetikzlibrary{matrix}
\usepackage{mathrsfs}
\usepackage{subcaption}
\usepackage{floatrow}

\newfloatcommand{capbtabbox}{table}[][\FBwidth]

\title[A Yee-like scheme on unstructured grids]{A Yee-like finite-element scheme for \\ Maxwell's equations on unstructured grids}

\author{B. Radu$^\dag$ \and H. Egger$^{\dag,*}$}
\address{\small$^\dag$Johann Radon Institute for Computational and Applied Mathematics, Linz, Austria\\
$^*$Institute for Computational Mathematics, Johannes-Kepler University Linz, Austria}
\email{bogdan.radu@ricam.oeaw.ac.at}
\email{herbert.egger@jku.at}

\definecolor{mygray}{rgb}{.5,.5,.5}

\newtheorem{lemma}{Lemma}[section]

\newtheorem{theorem}[lemma]{Theorem}
\newtheorem{corollary}[lemma]{Corollary}
\newtheorem{method}[lemma]{Method}

\theoremstyle{definition}
\newtheorem{remark}[lemma]{Remark}

\def\Th{\mathcal{T}_h}
\def\calE_h{\mathcal{E}_h}
\def\calE_hs{\mathcal{E}_h^\sigma}
\def\calE_hz{\mathcal{E}_h^0}
\def\EE{\mathscr{E}}
\def\dtautau{\partial_{\tau\tau}}
\def\dtau{\partial_{\tau}}
\def\dt{\partial_{t}}
\def\dtt{\partial_{tt}}
\def\dttt{\partial_{ttt}}
\def\dtttt{\partial_{tttt}}
\def\curl{\operatorname{curl}}

\def\wPi{\widetilde\Pi}

\def\mysigma{\delta}

\def\NC{\mathcal{N\mkern-2mu C\mkern-2mu}}
\def\N{\mathcal{N}}

\def\K{K}
\def\Th{\mathcal{T}_h}

\def\grad{\nabla}
\def\RR{\mathbb{R}}

\def\E{{E}}

\def\a{{a}}
\def\r{{r}}

\def\v{{v}}
\def\u{{u}}

\def\V{{V}}

\def\wt{\widetilde}
\def\wh{\widehat}

\def\ttE{\mathsf{E}}

\def\ttG{\mathsf{G}}
\def\ttK{\mathsf{K}}

\def\ttP{\mathsf{P}}
\def\ttR{\mathsf{R}}
\def\ttK{\mathsf{K}}

\def\ttM{\mathsf{M}}
\def\ttP{\mathsf{P}}

\def\ttP{\mathsf{P}}

\def\ttQ{\mathsf{Q}}

\def\ttv{\mathsf{v}}

\def\ttf{\mathsf{f}}

\def\ttF{\mathsf{F}}
\def\ttg{\mathsf{g}}

\def\eps{\varepsilon}

\def\calE{\mathcal{E}}

\newcommand{\mylabel}[2]{#2\def\@currentlabel{#2}\label{#1}}

\def\la{\langle}
\def\ra{\rangle}

\makeatletter
\newcommand{\leqnomode}{\tagsleft@true}
\newcommand{\reqnomode}{\tagsleft@false}
\makeatother

\newcommand{\vertiii}[1]{{\left\vert\kern-0.25ex\left\vert\kern-0.25ex\left\vert #1 
    \right\vert\kern-0.25ex\right\vert\kern-0.25ex\right\vert}}

\newcommand{\tnorm}{{\vert\kern-0.25ex\vert\kern-0.25ex\vert}}

\begin{document}

\begin{abstract}
A novel finite element scheme is studied for solving the time-dependent Maxwell's equations on unstructured grids efficiently. Similar to the traditional Yee scheme, the method has one degree of freedom for most edges and a sparse inverse mass matrix. This allows for an efficient realization by explicit time-stepping without solving linear systems. 
The method is constructed by algebraic reduction of another underlying finite element scheme which involves two degrees of freedom for every edge. Mass-lumping and additional modifications are used in the construction of this method to allow for the mentioned algebraic reduction in the presence of source terms and lossy media later on. 
A full error analysis of the underlying method is developed which by construction also  carries over to the reduced scheme and allows to prove convergence rates for the latter. 
The efficiency and accuracy of both methods are illustrated by numerical tests. 
The proposed schemes and their analysis can be extended to structured grids and in special cases the reduced method turns out to be algebraically equivalent to the Yee scheme. 
The analysis of this paper highlights possible difficulties in extensions of the Yee scheme to non-orthogonal or unstructured grids, discontinuous material parameters, and non-smooth source terms, and also offers potential remedies.
\end{abstract}

\maketitle

\begin{quote}

\end{quote}

\section{Introduction}
\label{sec:intro}

We consider the efficient numerical simulation of electromagnetic wave propagation through linear non-dispersive lossy media. As the basic mathematical model, we use Maxwell's equations in second-order form, i.e. 
\begin{align} 
\eps \dtt \E + \sigma \dt \E + \curl(\nu \curl \E) &= f,  \qquad \text{in }\Omega, \, t>0, \label{eq:maxwell1} \\
n \times (\nu \curl E) &= g, \qquad \text{on } \partial\Omega, \, t>0, \label{eq:maxwell2}
\end{align}
Here $\eps$ and $\sigma$ denote the electric permittivity and conductivity, and $\nu$ the magnetic reluctivity of the medium; further $E$ is the electric field intensity,  $f$, $g$ describe the excitation by volume or surface currents, and $\Omega$ is the computational domain.

One of the standard approaches for solving \eqref{eq:maxwell1}--\eqref{eq:maxwell2} numerically is the Yee scheme and its relatives, e.g., the finite difference time domain (FDTD) method and the finite-integration technique (FIT); see \cite{Yee66}, \cite{Taflove80,Taflove05}, and \cite{Weiland77,Weiland03} for details.
The common algebraic form of all these methods reads 
\begin{align} \label{eq:yee}
\ttM_\eps \frac{\ttE^{n+1} - 2 \ttE^n + \ttE^n}{\tau^2} + \ttM_\sigma \frac{\ttE^{n+1} - \ttE^{n-1}}{2\tau} + \ttK_\nu \ttE^n &= \ttf^n + \ttg^n.
\end{align}
This allows for an efficient realization by explicit time stepping, whenever the mass matrices $\ttM_\eps$, $\ttM_\sigma$ are diagonal. On structured orthogonal grids and for homogeneous material distributions, such constructions are possible and lead to second-order accurate approximations in space and time; see e.g. \cite{Cohen02,Taflove05,Yee66}. These results remain valid for non-uniform orthogonal grids \cite{MonkSuli94}; further see \cite{Cohen02,CohenMonk95} for related higher-order methods.
The extension of the Yee scheme to discontinuous material parameters or source terms and to non-orthogonal or even unstructured grids is, however, not straightforward; see \cite{Holland83,Lee92,SchuhmannWeiland98,SchuhmannWeiland98b} and \cite{BossavitKettunen99,CodecasaKapidaniSpecognaTrevisan18,CodecasaPoliti08} for some attempts in the latter case. To the best of our knowledge, a rigorous convergence analysis of Yee-like schemes in such situations is not available up to date.

An alternative approach that allows the construction of stable schemes for structured as well as unstructured grids, and also for discontinuous material parameters, is offered by finite element discretizations. A rigorous error analysis for various methods can be developed; see e.g. \cite{Monk92a,Monk92c,Monk93}; we refer to \cite{Cohen02,Joly03} for an overview of different approaches, their analysis, and further references. 
Together with the leap-frog time-stepping scheme, the finite element approximation of \eqref{eq:maxwell1}--\eqref{eq:maxwell2} again leads to algebraic systems of the form \eqref{eq:yee}. 
In contrast to the finite difference schemes mentioned above, the mass matrices $\ttM_\eps$ and $\ttM_\sigma$ are, however, usually not diagonal, and the realization of \eqref{eq:yee} thus requires the solution of linear systems in every time step. Therefore, the use of implicit time-stepping schemes has been advocated in the literature \cite{MakridakisMonk95,Monk93a}. 

This inherent disadvantage of finite element approximations can be overcome by mass-lumping, which aims at replacing $\ttM_\eps$, $\ttM_\sigma$ in \eqref{eq:yee} by diagonal or block-diagonal matrices. For orthogonal grids, corresponding schemes have been proposed in \cite{CohenMonk95} and modifications for unstructured grids have been considered in \cite{ElmkiesJoly97b,ElmkiesJoly97}; we also refer to \cite{Cohen02} for a detailed discussion. A full convergence analysis for a second order method has been given recently in \cite{EggerRadu21a_maxwelltet,Radu22},  and in \cite{KapidaniCodecasaSchoberl21} higher order approximations of Maxwell's equations were obtained based on staggered grid complexes. 
Another approach for generating non-conforming finite element approximations with block-diagonal mass matrices is provided by discontinuous-Galerkin methods \cite{CohenPernet16,HesthavenWarburton08}. For low-order approximations however, these methods suffer from a substantial increase in the number of degrees of freedom; we refer to \cite{GeeversMulderVegt18} for comparison with mass-lumping schemes in the context of elastodynamics. 

\subsection*{Main contributions}
In this paper, we propose and analyze a fully discrete finite element approximation for \eqref{eq:maxwell1}--\eqref{eq:maxwell2} which can be considered as a natural extension of the Yee scheme to unstructured grids. This method involves only one degree of freedom for most edges and it can be realized efficiently as an explicit time-stepping scheme; moreover, a full convergence analysis is possible.
The approach and its analysis are based on the following key ideas, already presented in \cite{EggerRadu20c_maxwellyee}: 
First, a finite element method is considered involving two degrees of freedom on each edge. Mass lumping is achieved by appropriate numerical quadrature and a corresponding choice of basis functions, and a full error analysis can be developed based on standard arguments.
In a second step, the method is then reduced on the algebraic level to a scheme that involves only one degree of freedom for most edges. While a direct variational characterization of this scheme is no longer possible, its close relation to the first method nevertheless allows to develop a rigorous error analysis.
Compared to \cite{EggerRadu20c_maxwellyee}, some non-trivial modifications are required in the construction and the analysis of the proposed methods to deal with non-trivial conductivities $\sigma$ and inhomogeneous problem data $f$, $g$.
On orthogonal grids and for special situations, the reduced method becomes equivalent to the traditional Yee scheme. The analysis of this paper therefore also offers some recipes for extensions of the Yee scheme to non-orthogonal grids and discontinuous parameters and problem data.
The algebraic form of the numerical scheme obtained after reduction also shares similarities with the approaches of \cite{CodecasaKapidaniSpecognaTrevisan18,CodecasaPoliti08}, which will be briefly discussed at the end of the manuscript.  

\subsection*{Outline}

In Section~\ref{sec:prelim}, we introduce our notation and basic assumptions. 
The methods and main results are presented in Section~\ref{sec:main}. 
The error analysis is developed in Section~\ref{sec:proofs} and the implementation of the method is discussed in some detail in Section~\ref{sec:implementation}.
For an illustration of our theoretical results, some computational tests are presented in Section~\ref{sec:numerics}, and we close with a short discussion.

\section{Notation and basic assumptions} 
\label{sec:prelim}

We consider a three-dimensional setting and assume that $\Omega \subset \RR^3$ is a bounded polyhedral Lipschitz domain.
Throughout the paper, we assume that 
\begin{enumerate}[topsep=1em]
\item [(A1)] %
$\Th$ is a geometrically-conforming non-overlapping partition of $\Omega$ into tetrahedral elements $\K \in \Th$; the mesh $\Th$ is shape-regular and quasi-uniform, i.e., $\gamma h_\K^3 \le |\K| \le h_\K^3$ and $\delta h \le h_\K \le h$ for all $\K \in \Th$ with $\gamma,\delta>0$. 
\end{enumerate}
As usual, $h_\K$ and $|\K|$ denote the diameter and the volume of the element $\K \in \Th$, and $h=\max_\K h_\K$ is the global mesh size; see \cite{ErnGuermond} for further details.
All results are presented in detail for the three-dimensional setting and unstructured grids, but the translation to two dimensions and structured grids is possible and more or less straightforward; see \cite[Chapter 5]{Radu22} and Section~\ref{sec:numerics}.

The material parameters and problem data are required to be sufficiently regular and to satisfy the usual physical bounds. Moreover, material discontinuities shall be resolved by the mesh. 
For ease of presentation, we thus assume that
\begin{enumerate}[topsep=1em]
\item[(A2)]  %
$\eps$, $\sigma$, $\nu \in P_0(\Th)$ and $0 \le \eps,\sigma,\nu \le \overline c$ and $\eps,\nu \ge \underline c$ for some constants $\underline c, \overline c>0$. \\[-1ex]
\item[(A3)] %
$f : [0,T] \to L^2(\Omega)^3$ and $g: [0,T] \to L^2(\partial\Omega)^3$ are smooth functions of time, with $f(0)=0$ and $g(0)=0$, and the initial conditions are $E(0)=\dt E(0)=0$. 
\end{enumerate}
We write $P_k(\Th)=\{v : v|_K \in P_k(K), \ \forall K \in \Th\}$ for the space of piecewise polynomials of degree $k$ on $\Th$.
Under these assumptions, the existence of a unique solution to \eqref{eq:maxwell1}--\eqref{eq:maxwell2} can be established by standard arguments; see e.g. \cite{Leis88,Monk92a}. 
Also, more general conditions could be treated with minor modifications to our analysis.

\subsection*{Function spaces}

We use standard symbols $L^2(\Omega)$, $H^1(\Omega)$, and $H(\curl;\Omega)$ for the spaces of square-integrable function with square integrable weak gradients and curls, respectively; see \cite{Monk03} for details. 
The norm of a space $X$ is denoted by $\|\cdot\|_{X}$ and we will often write $\la a,b\ra = \int_\Omega a \cdot b \, dx$ and $\la a,b\ra_{\partial}=\int_{\partial\Omega} a \cdot b \, ds(x)$ for the scalar products of two functions in $L^2(\Omega)^3$ and $L^2(\partial\Omega)^3$, respectively. 
For the error analysis, we use $H^k(\Th)=\{v \in L^2(\Omega) : v|_\K \in H^k(\K), \ \forall \K \in \Th\}$, to denote spaces of piecewise smooth functions, and we write $\|v\|_{H^k(\Th)} = (\sum_\K \|v\|_{H^k(\K)}^2)^{1/2}$ for the corresponding norms. 
We further denote by $L^p(0,T;X)$ the Bochner spaces of functions with values in $X$ whose $p$-th power is integrable in time. The functions in $W^{k,p}(0,T;X)$ further have weak derivatives in $L^p(0,T;X)$. 
For brevity, we will sometimes write $L^p(X)$ and $W^{k,p}(X)$, and omit explicit reference to the time interval.

\subsection*{Space discretization}

For the spatial approximation of the electric field $E$, we first consider Nédélec finite elements of type II, i.e.
\begin{align}
V_h = \{v_h \in H(\curl;\Omega) : v_h|_\K \in P_1(\K)^3 \ \forall \K \in \Th\}.
\end{align}
This amounts to the space of piecewise linear vector-valued functions with tangential continuity across element boundaries; see \cite{BoffiBrezziFortin13,Monk03,Nedelec86}. 
Let us recall that functions in $V_h$ have two degrees of freedom for every edge, and the canonical interpolation operator for the space is given by $\Pi_h : H(\curl;\Omega) \cap H^1(\Th)^3 \to V_h$, with 
\begin{align}
\int_e \Pi_h v \cdot \tau_e \, p_e \, ds &= \int_e v \cdot \tau_e  \, p_e \, ds \qquad \forall p_e \in P_1(e), \ e \in \calE_h. 
\end{align}
Here $\calE_h = \{e_{ij} : i<j\}$ is the set of edges $e_{ij}=(v_i,v_j)$ and $\tau_e$ is the unit tangential vector on $e=e_{ij}$ pointing from vertex $v_i$ to $v_j$ with $i<j$. 

For a given subset $\wt\calE_h \subset \calE_h$ of edges, we define the corresponding subspace
\begin{align}
\wt V_h=\{v_h \in V_h : v_h \cdot \tau_e \in P_0(e) \ \forall e \in \wt\calE_h\} \subset V_h,
\end{align}
which consists of functions in $V_h$ having only constant tangential trace, and therefore only one degree of freedom for edges $e \in \wt\calE_h$. We thus call $\wt V_h$ the \emph{reduced space} in the following. 
The canonical interpolation operator for  $\wt V_h$ is given by $\wt \Pi_h : H^1(\Th)^3 \cap H(\curl;\Omega) \to \wt V_h$ with
\begin{align}
\int_e \wt \Pi_h v \cdot \tau_e \, p_e \, ds &= \int_e v \cdot \tau_e  \, p_e \, ds \qquad \forall p_e \in P_{k_e}(e), \ e \in \calE_h, 
\end{align}
and with degree $k_e=0$ for $e \in \wt\calE_h$ and $k_e=1$ for $e \in \calE_h \setminus\wt\calE_h$.
For any choice of $\wt\calE_h \subset\calE_h$, we have the inclusions
\begin{align*}
\N_0(\Th) \cap H(\curl;\Omega) \subset \wt V_h \subset  \NC_1(\Th) \cap H(\curl;\Omega).
\end{align*}
Here $\N_0$ and $\NC_1$ denote the lowest order Nédélec elements of type I and II. This ensures good approximation properties for both spaces; see \cite{BoffiBrezziFortin13,Nedelec80,Nedelec86} for details. 
For the choice $\wt\calE_h=\calE_h$ or $\wt\calE_h=\emptyset$, one of the two inclusions becomes an identity.

\subsection*{Mass lumping}

For the approximation of some of the integrals arising in the finite element approximation of \eqref{eq:maxwell1}--\eqref{eq:maxwell2}, we use numerical integration by the vertex rule; this will allow for mass-lumping later on.
For ease of notation, we introduce
\begin{align}
\la a,b\ra_h = \sum\nolimits_T   \tfrac{|T|}{4} \sum\nolimits_{v_i \in T} a(v_i) \cdot b(v_i) . 
\end{align}
Here $a,b$ are assumed to be piecewise smooth vector-valued functions over the mesh~$\Th$.
Let us note that the quadrature is exact if $a \cdot b \in P_1(\Th)$.

\subsection*{Time discretization}
 
Let $\tau=T/N$ and $t^n = n \tau$ be a sequence of uniformly spaced time steps. Further let $(a^n)_{n \ge 0} \subset X$ be a sequence in some vector space $X$. Then 
\begin{align}
\dtautau a^n := \frac{a^{n+1} - 2 a^n + a^{n-1}}{\tau^2} 
\qquad \text{and} \qquad 
\dtau a^{n-1/2} := \frac{a^n-a^{n-1}}{\tau},
\end{align}
are used to denote the standard central difference quotients approximating the second and first derivative at time $t=t^n$ and $t=t^{n-1/2} = t^n-\frac{1}{2} \tau$, respectively.

\section{Main results}
\label{sec:main}

For the numerical approximation of \eqref{eq:maxwell1}--\eqref{eq:maxwell2} with homogeneous initial conditions, see assumption (A3), we now consider the following fully discrete scheme.
\begin{method} \label{meth:main}
Let $\wt\calE_h \subset \calE_h$ and $\wt \Pi_h : V_h \to \wt V_h$ denote the appropriate projection. 
Find $E_h^n \in V_h$, $0 \le n \le N$, with $E_h^0=E_h^1=0$ and such that
\begin{align} \label{eq:method}
\la(\eps + \tfrac{\tau}{2}\sigma) \dtautau E_h^n, v_h\ra_h + \la\sigma \wt \Pi_h & \dtau E_h^{n-1/2},\wt\Pi_h v_h\ra_h + \la\nu \curl E_h^n,\curl v_h\ra \\
&= \la f(t^n),v_h\ra + \la g(t^n),v_h\ra_{\partial\Omega}, \qquad \forall \ 1 \le n < N. \notag
\end{align}
The solution (sequence) will be abbreviated by the symbol $E_h = (E_h^n)_{0 \le n \le N}$.
\end{method}
Independently of the choice of the set of edges $\wt \calE_h \subset \calE_h$, on which the number of degrees of freedom is reduced, the implementation of the method leads to a finite-dimensional recursion of the form
\begin{align} \label{eq:impl}
\ttM_{\eps + \tau \sigma/2} \frac{\ttE^{n+1} - 2 \ttE^n + \ttE^{n-1}}{\tau^2} + \widehat \ttM_{\sigma} \frac{\ttE^n - \ttE^{n-1}}{\tau} + \ttK_\nu \ttE^n 
&= \ttf^n + \ttg^n.
\end{align}
The well-posedness of the discretization scheme then follows immediately from the regularity of the matrix $\ttM_{\eps + \tau \sigma/2}$, which is a direct consequence of Lemma~\ref{lem:equiv} below. 
In Section~\ref{sec:implementation}, we further show that an appropriate choice of basis functions for the space $V_h$, adopted to the numerical quadrature, leads to a block-diagonal mass matrix $\ttM_{\eps + \tau \sigma/2}$, such that time stepping in \eqref{eq:impl} can be realized efficiently. 

The algebraic form \eqref{eq:impl} reveals that Method~\ref{meth:main} is based on an explicit time-stepping scheme and a restriction on the time step size $\tau$ is, therefore, required to ensure discrete stability and convergence with $h,\tau \to 0$. 
We thus assume that 
\begin{enumerate}[topsep=1em]
\item[(A4)] \label{ass:A4} 
the time step $\tau>0$ is chosen to satisfy for all $\v_h \in \V_h$ the inequality
\begin{align*}
   \tfrac{{\tau}^2}{4}\la\nu\curl\v_h,\curl \v_h\ra + \tfrac{\tau}{2} \left|\la\sigma\wt\Pi_h \v_h,\wt\Pi_h \v_h\ra - \la\sigma\v_h,\v_h\ra\right| \le \tfrac{1}{2} \la\varepsilon\v_h,\v_h\ra_h.
\end{align*}
\end{enumerate}
For conductivity $\sigma=0$ and $\la\cdot,\cdot\ra_h=\la\cdot,\cdot\ra$, this assumption reduces to the usual CFL condition as used, e.g., in \cite{Cohen02,Joly03}.
Under our assumptions on the mesh and the model parameters, one can verify that $\tau\le Ch$ for some appropriate constant $C>0$ is sufficient to guarantee condition (A4); see Section~\ref{sec:numerics}.
In practice, an appropriate time step $\tau$ satisfying (A4) can be found by performing a few vector iterations. 

To guarantee good approximation properties, we further need some restriction on the set $\wt \calE_h$ of edges, on which the polynomial order is reduced.
We thus require that
\begin{enumerate}[topsep=1em]
\item[(A5)] \label{ass:A5}
$\sigma$ is continuous across edges $e \in \wt \calE_h$ inside $\Omega$, and $\sigma=0$ for all $e \in \wt \calE_h$ on $\partial\Omega$.
\end{enumerate}
This condition simply means that we stay with two degrees of freedom on edges where the conductivity $\sigma$ is either discontinuous or non-trivial at the boundary. 
The reason for this restriction will become clear from the error analysis given in the next section and its necessity will be illustrated by numerical tests.

For ease of notation, we write $\|u_h\|_{\ell_\infty(X)} = \max_{0 \le k \le N-1} \|u_h^{n+1/2}\|_X$ in the following statements, and we write $u^{n+1/2}=u(t^{n+1/2})$ for functions $u$ that are continuous in time.
We further denote by $\wh u^{\,n+1/2} = \frac{1}{2}(u^{n+1}+u^n)$ the average at intermediate time steps. 
This allows us to present our first main result as follows. 
\begin{theorem} \label{thm:main1}
Let $E$ be a sufficiently smooth solution of \eqref{eq:maxwell1}--\eqref{eq:maxwell2} and let (A1)--(A5) hold.
Then Method~\ref{meth:main} is well-defined and the discrete solution $E_h=(E_h^n)_n$ satisfies
\begin{align*}%
\|\dt E - \dtau E_h\|_{\ell_\infty(L^2(\Omega))} &+ \|\curl (E - \widehat E_h)\|_{\ell_\infty(L^2(\Omega))}\le C(E) h + C'(E) \tau^2, 
\end{align*}
with constants
\begin{align*}
C(E) &= \|\dt E\|_{L^\infty(H^1(\Th))} + \|\curl E\|_{L^\infty(H^1(\Th))}  + \|\dtt E\|_{L^\infty(H^1(\Th))} \\
&\qquad \qquad \qquad %
+ \|\curl\dt E\|_{L^1(H^1(\Th))} + \|\dt E\|_{L^1(H^1(\Omega))}, \\
C'(E) &= \|\dtttt E\|_{L^1(L^2(\Omega))} + \|\dttt E\|_{L^1(L^2(\Omega))} + \|\curl\dtt E\|_{L^1(H^1(\Th))} . \qquad
\end{align*}
The implementation leads to a time-stepping scheme \eqref{eq:impl}, and for an appropriate choice of a basis, the matrix $\ttM_{\eps + \tau\sigma/2}$ is block-diagonal, while $\widehat \ttM_{\sigma}$ and $\ttK_\nu$ are sparse.
\end{theorem}
The convergence results is proven in  Section~\ref{sec:proofs}, while the algebraic structure of the scheme is derived in Section~\ref{sec:implementation}.
Let us note that the assertions hold, in particular, for the choice $\wt\calE_h = \emptyset$, for which the projection $\wt\Pi_h$ drops out and the method as well as its implementation become somewhat simpler; see \cite{Radu22} for details.

\subsection*{The Yee-like scheme}

The reason for introducing the projection $\wt\Pi_h$ in Method~\ref{meth:main} lies in the following important observation, which is summarized as our second main result and leads to the reduced Yee-like scheme announced in the introduction.
\begin{theorem} \label{thm:main2}
Let the assumptions of Theorem~\ref{thm:main1} be valid and $E_h=(E_h^n)_n$ denote the solution of Method~\ref{meth:main}. 
Further define $\wt E_h^n = \wt \Pi_h E_h^n$ for all $n \ge 0$. 
Then 
\begin{align*}%
\|\dt E - \dtau \wt E_h\|_{\ell_\infty(L^2(\Omega))} &+ \|\curl (E - \widehat {\wt E}_h)\|_{\ell_\infty(L^2(\Omega))}\le C(E) h + C'(E) \tau^2, 
\end{align*}
with constants $C(E)$ and $C'(E)$  of the same form as in Theorem~\ref{thm:main1}.
Moreover, the coefficients of the solution $\wt E_h$ can be computed by the time-stepping scheme
\begin{align}\label{eq:impl2}
\dtautau \wt \ttE^n 
&= \wt \ttM_{\eps + \tau \sigma/2}^{-1} (- \wt \ttM_\sigma \dtau \wt \ttE^{n-1/2} - \wt \ttK_\nu \wt \ttE^n) + \wt \ttF^n + \wt \ttG^n,
\end{align}
and for an appropriate choice of basis for $\wt V_h$, the matrices $\wt\ttM_{\eps + \tau\sigma/2}^{-1}$, $\wt \ttM_{\sigma}$ and $\wt \ttK_\nu$ are sparse and the vectors $\wt\ttF^n$ and $\wt\ttG^n$ can be cheaply assembled from $\ttf^n$ and $\ttg^n$ in \eqref{eq:impl}.
\end{theorem}
\begin{remark}
If $\sigma=0$,
we may choose $\wt\calE_h =\calE_h$ and obtain an explicit time-stepping method for Maxwell's equations with exactly one degree of freedom per edge.
For orthogonal grids and homogeneous data $f,g\equiv 0$, the presented approach becomes equivalent to the Yee scheme; see \cite{EggerRadu20c_maxwellyee} for details. 
\end{remark}
The two assertions of Theorem~\ref{thm:main2} are again proven in the following two sections. 
\begin{remark}
Let us emphasize that the matrix $\wt \ttM_{\eps + \tau \sigma/2}$ in \eqref{eq:impl2} has a sparse inverse, but it is not a sparse matrix by itself. 
In contrast to \eqref{eq:impl}, which corresponds to Method~\ref{meth:main}, we can not give a variational characterization of the scheme \eqref{eq:impl2} in closed form. 
This poses a severe challenge for the analysis of this method which can be overcome only by a somewhat non-standard analysis. 
\end{remark}

\section{Proof of convergence rates}\label{sec:proofs}

In this section, we establish the convergence rates stated in Theorem~\ref{thm:main1} and \ref{thm:main2}. To be able to do so, we require a couple of auxiliary results, which are stated first.

\subsection{Projection operators} 

For later reference, we collect some well-known properties of projection operators arising in our error analysis below. 
\begin{lemma}
\label{lem:proj}
Let (A1) hold, $\wt\calE_h \subset \calE_h$, and $\Pi_h$, $\wt \Pi_h$ be defined as in Section~\ref{sec:prelim}. Then 
\begin{alignat}{2}\label{eq:projest1}
\begin{split}
\|\E-\wPi_h\E\|_{L^2(\K)}&\leq C h \|\E\|_{H^{1}(\K)}, \\
\|\curl(\E-\wPi_h \E)\|_{L^2(\K)} &\le C h \|\curl\E\|_{H^{1}(\K)},
\end{split}
\end{alignat}
for all $E \in H^1(\Th)$ and $K \in \Th$ with a constant $C$ depending only on $\gamma,\delta$ in assumption~(A1).
The same estimates also hold for the projection operator $\Pi_h$.
\end{lemma}
The proof of the assertions follows from the arguments given in \cite[Sec.~2.5]{BoffiBrezziFortin13}.
In our analysis, we will also make use of the $L^2$-orthogonal projection $\pi^0_\omega : L^2(\omega) \to P_0(\omega)$ to constants for certain subsets $\omega \subset \Omega$, which is defined by %
\begin{align} \label{eq:l2proj}
\int_\omega \pi^0_\omega v \, dx = \int_\omega v \, dx. 
\end{align}
The same symbol will also be used for the projection of vector-valued functions. 
\begin{lemma} \label{lem:proj2}
Let (A1) hold and $\omega=K$ for $K \in \Th$ or $\omega=\bigcup_{K \cap e =e} K$ for $e \in \calE_h$. Then the projection error can be estimated by
\begin{align}
\|v - \pi^0_\omega v\|_{L^2(\omega)} \le C h \|v\|_{H^1(\omega)}
\end{align}
for all $v \in H^1(\omega)$ with constant $C$ depending only on $\gamma,\delta$ in assumption (A1). 
\end{lemma}
The proof of this assertion is based on the Poicar\'e--Friedrichs inequality \cite[Thm.~1.1]{FarwigRosteck2016} and standard scaling arguments; also see~\cite[Ch.~4]{BrennerScott94}. 
Here we use that the sets $\omega$ appearing in the lemma are uniformly star-shaped with respect to balls of size $h$, which follows from assumption (A1) on the mesh, and hence $C$ is universal.  

We will further write $\pi^0_h : L^2(\Omega) \to P_0(\Th)$ for the projection to piecewise constants over the mesh $\Th$, defined by $(\pi^0_h v)|_K = \pi^0_\K(v|_K)$, and note that 
\begin{align} \label{eq:proj0}
\|v - \pi^0_h v\|_{L^2(K)} \le C h \|v\|_{H^1(K)},
\end{align}
which follows immediately from the assertion of the previous lemma.

\subsection{Properties of the quadrature rule}

As a second ingredient, we now state some elementary facts about the quadrature rule introduced in Section~\ref{sec:prelim}. 
\begin{lemma} \label{lem:equiv}
Let (A1) hold and $\alpha \in P_0(\Th)$ with $\alpha \ge 0$. Then 
\begin{align}
c\,\la\alpha v_h,v_h\ra \le \la\alpha v_h,v_h\ra_h \le C \, \la\alpha v_h,v_h\ra \quad \forall v_h\in V_h,
\end{align}
with uniform constants $c,C>0$ depending only on the bounds in assumption (A1). 
\end{lemma}
\begin{proof}
We consider a single element $K \in \Th$ and abbreviate $\la a,b\ra_K=\int_K a \cdot b \, dx$ and $\la a,b\ra_{h,K} = \frac{|K|}{4} \sum_i a(v_i) \cdot b(v_i)$. By mapping to the reference element, using the finite dimensionality of $P_1(K)^3$, and noting that $\alpha \ge 0$ is piecewise constant, we get 
\begin{align*}
c \, \la\alpha v_h,v_h\ra_K \le \la\alpha v_h,v_h\ra_{h,K} \le C \, \la\alpha v_h,v_h\ra_K,
\end{align*}
with constants $c,C$ that only depend on the shape regularity of the element $K$. 
The assertion of the lemma then follows by summation over all elements $K \in \Th$.
\end{proof}
As a direct consequence of the previous result, we obtain the following assertions:
\begin{corollary} \label{cor:equiv}
Let (A1) hold.
Then $\|v_h\|_h^2=\la v_h,v_h\ra_h$ defines a norm on $V_h$ and
\begin{align}
c_1 \|v_h\|_{L^2(\Omega)} \le \|v_h\|_h^2 \le c_2 \|v_h\|_{L^2(\Omega)}^2 \quad \forall v_h \in V_h.
\end{align}
Further let $\alpha \in P_0(\Th)$ with $0 < \underline \alpha \le \alpha \le \overline \alpha$. 
Then $\la\alpha u_h,v_h\ra_h$ defines a continuous and elliptic symmetric bilinear form on $V_h$, more precisely
\begin{align*}
c_1 \underline \alpha \|u_h\|_{L^2(\Omega)}^2 \le \la\alpha u_h,u_h\ra_h \quad \text{and} \quad 
\la\alpha u_h,v_h\ra_h \le c_2 \overline \alpha \|u_h\|_{L^2(\Omega)} \quad \forall u_h,v_h \in V_h.
\end{align*}
\end{corollary}
These properties immediately imply the well-posedness of Method~\ref{meth:main}.
As a next ingredient for our analysis, we analyze the quadrature error. 
\begin{lemma}\label{lem:quaderror}
Let assumption (A1) hold and let 
\begin{align*}
\mysigma_{h}(\alpha u_h,v_h)\coloneqq \la\alpha u_h,v_h\ra_h - \la\alpha u_h,v_h\ra
\end{align*}
denote the quadrature error for some $\alpha \in P_0(\Th)$. 
Then
\begin{align*}
|\mysigma_{h}(\alpha \wPi_h \u,\v_h)| \le C h \|u\|_{H^1(\Th)}\|v_h\|_{L^2(\Omega)}
\qquad \forall u \in H^1(\Th)^3, v_h \in V_h,
\end{align*}
with constant $C$ depending only on $\|\alpha\|_{L^\infty(\Omega)}$ and the constants in assumption (A1).
\end{lemma}
\begin{proof}
We define the local error $\mysigma_{h,\K}(\alpha u_h,v_h) \coloneqq \la\alpha u_h,v_h\ra_{h,K}-\la\alpha u_h,v_h\ra_K$, and split
\begin{align*}
|\mysigma_{h,\K}(\alpha \wPi_h \u,\v_h)| &\le |\alpha \mysigma_{h,\K}(\pi^0_h\u,\v_h)| + |\alpha \mysigma_{\K}(\wPi_h \u-\pi^0_h\u,v_h)| = (i) + (ii),
\end{align*}
where we used that $\alpha$ is piecewise constant. 
Since the quadrature rule integrates linear polynomials exactly, we obtain $(i)=0$. 
The second term can again be bounded elementwise. We may therefore omit $\alpha$ and obtain 
\begin{align*}
|\mysigma_{\K}(\wPi_h \u-\pi^0_h\u,\v_h)|&\le (1+c_1) \|\wPi_h \u-\pi^0_\K\u\|_{L^2(\K)}\|\v_h\|_{L^2(\K)}\\
&\le C \big(\|\u-\pi^0_h\u\|_{L^2(\K)} + \|\u-\wPi_h \u\|_{L^2(\K)}\big)\|\v_h\|_{L^2(\K)} \\
& \le C'h\, \|u\|_{H^1(\K)}\|\v_h\|_{L^2(\K)}.
\end{align*}
Here we used the assertions of Corollary~\ref{cor:equiv} and the Cauchy-Schwarz inequality in the first step, and the projection error estimates of Lemma~\ref{lem:proj} and \ref{lem:proj2} in the last.
Scaling by the constant $\alpha$ and summation over all elements leads to the assertion.
\end{proof}

\subsection{Estimates for the loss term}

We now present a particular approximation property, which explains why the degree of approximation can be reduced on the edges in the set $\wt\calE_h$ satisfying assumption (A5) without decreasing the accuracy. 
\begin{lemma} \label{lem:basis}
Let $\lambda_i \in P_1(\Th) \cap H^1(\Omega)$ denote the barycentric coordinates defined by $\lambda_i(v_j)=\delta_{ij}$ for all vertices $v_j$ of the mesh. 
To every edge $e_{ij} \in \calE_h$, we define \begin{align} \label{eq:basis}
\Phi_{ij} = \lambda_i \nabla \lambda_j \qquad \text{and} \qquad \Phi_{ji} = -\lambda_j \nabla \lambda_i. 
\end{align}
These basis functions are linearly independent and  $V_h = \operatorname{span}\{\Phi_{ij},\Phi_{ji} : e_{ij} \in \calE_h\}$. 
Hence the functions defined in \eqref{eq:basis} 
comprise a basis for $V_h$.
\end{lemma}
The assertion follows immediately from the considerations in \cite{BoffiBrezziFortin13,Nedelec86}. 
In the subsequent analysis, we will also make use of the following norm equivalence.
\begin{lemma}\label{lem:normeq}
Let (A1) hold and let $v_h \in V_h$, i.e., $v_h=\sum_{e_{ij}\in\calE_h}\ttv_{ij}\Phi_{ij}+\ttv_{ji}\Phi_{ji}$ for appropriate coefficients $\ttv_{ij},\ttv_{ji} \in \RR$.
Further let $\omega(e_{ij}) = \bigcup_{K \cap e_{ij} = e_{ij}} K$ denote the patch of elements containing the edge $e_{ij}$. 
Then
\begin{align*}
\vertiii{v_h}_h^2\coloneqq \sum\nolimits_{e_{ij}\in\calE_h} \ttv_{ij}^2\|\Phi_{ij}\|^2_{L^2(\omega(e_{ij}))} + \ttv_{ji}^2\|\Phi_{ji}\|^2_{L^2(\omega(e_{ij}))}
\end{align*}
defines a norm on $V_h$ which is equivalent to $\|\cdot\|_{L^2(\Omega)}$. More precisely, one has 
\begin{align*}
c_1' \|v_h\|_{L^2(\Omega)} \le \vertiii{v_h}_h \le c_2' \|v_h\|_{L^2(\Omega)} \quad \forall v_h \in V_h
\end{align*}
with uniform constants $c_1',c_2'>0$ depending only on the bounds in assumption (A1).
\end{lemma}
The result again follows by scaling arguments and the equivalence of norms on finite-dimensional spaces.
Using these observations, we can now prove the following.
\begin{lemma}\label{lem:projprop}
Let (A1)--(A2) and (A5) hold. 
Then for any $u \in H^1(\Omega)^3$, we have
\begin{align}\label{eq:projprop}
\la\sigma\wPi_h u,\wPi_h\v_h-\v_h\ra \le C \,  h \, \|u\|_{H^1(\Omega)}\|v_h\|_{L^2(\Omega)} \qquad \forall v_h\in V_h,
\end{align}
with a uniform constant $C$ depending only on the bounds in the assumptions. \end{lemma}
\begin{proof}
We start by considering a special test function $v_h = \Phi_{ij}$, where $\Phi_{ij}$ is one of the basis functions introduced in Lemma~\ref{lem:basis}. 
We then split 
\begin{align}
\la\sigma\wPi_h u,\wt\Pi_h\Phi_{ij}-\Phi_{ij}\ra
&= \la\sigma\pi_\omega^0 u,\wt\Pi_h\Phi_{ij}-\Phi_{ij}\ra + \la\sigma(\wPi_h u-\pi^0_\omega u),\wt\Pi_h\Phi_{ij}-\Phi_{ij}\ra  \notag \\
&=(i)+(ii), \label{eq:splitting}
\end{align}
where $\pi^0_\omega:L^2(\omega)\to P_0(\omega)$ is the $L^2$-projection onto constants on the support $\omega=\omega(e_{ij})$ of the basis function $\Phi_{ij}$. 
If $e_{ij}\in \calE_h \setminus \wt \calE_h$, we have $\Phi_{ij} = \wPi_h\Phi_{ij}$, which means that $(i)=0$ in this case. 
If $e_{ij}\in \wt\calE_h$, on the other hand, then we deduce from assumptions (A2) and (A5) that $\sigma$ is constant on the patch $\omega=\omega(e_{ij})$.  
We can thus find a vector $b_1\in P_1(\omega)^3$ such that
$\curl b_1 = \sigma\pi^0_\omega u$,
and evaluate
\begin{align*}
(i) &= \la\curl b_1,\wt\Pi_h\Phi_{ij}-\Phi_{ij}\ra_\omega \\
&=\la b_1,\curl(\wPi_h\Phi_{ij}-\Phi_{ij})\ra_\omega + \la b_1,n\times(\wPi_h\Phi_{ij}-\Phi_{ij})\ra_{\partial\omega}
=(iii)+(iv).
\end{align*}
By elementary computations, see Lemma~\ref{lem:alg_proj} below, one can verify that
\begin{align}\label{eq:curlcommuting}
\wPi_h\Phi_{ij}-\Phi_{ij} = \grad(\lambda_i\lambda_j), \text{ and hence  }
\curl(\wPi_h\Phi_{ij}-\Phi_{ij})=0,
\end{align}
which in turn implies $(iii)=0$. 
By definition of the basis functions, the tangential components of $\Phi_{ij}$ and also that of $\wt\Pi_h \Phi_{ij}$ vanish on $\partial\omega(e_{ij})$, unless $e_{ij}$ is a boundary edge, which is excluded by assumption (A5).
Hence $(iv)=0$, and as a consequence, we see that $(i)=0$. 
Since we assumed $u\in H^1(\Omega)$, we may further bound
\begin{align*}
(ii) \le C \, h \, \|u\|_{H^1(\omega)}\|\Phi_{ij}\|_{L^2(\omega)}
\end{align*}
by employing the estimates of Lemma~\ref{lem:proj} and \ref{lem:proj2}. 
In summary, we thus have
\begin{align*}
\la\sigma\wPi_h u,\wt\Pi_h\Phi_{ij}-\Phi_{ij}\ra\le C \, h \, \|u\|_{H^1(\omega)}\|\Phi_{ij}\|_{L^2(\omega)}.
\end{align*}
The same estimate is obtained for the basis functions $\Phi_{ji}$. 
By Lemma~\ref{lem:basis}, any test function in $V_h$ can be expanded as $v_h=\sum_{e_{ij}}\ttv_{ij}\Phi_{ij}+\ttv_{ji}\Phi_{ji}$, and by splitting and summing over all elements, we immediately obtain
\begin{align*}
\la\sigma\wPi_h u,\wPi_h\v_h-\v_h\ra
&\le C  h \sum\nolimits_{e_{ij}} \|u\|_{H^1(\omega(e_{ij}))} \left(\ttv_{ij}\|\Phi_{ij}\|_{L^2(\omega(e_{ij}))}+\ttv_{ji}\|\Phi_{ji}\|_{L^2(\omega(e_{ij}))}\right) \\
&\le C h \Big(\sum\nolimits_{e_{ij}}\|u\|^2_{H^1(\omega(e_{ij}))}\Big)^{1/2} \vertiii{v_h}_h \le c'' h \|u\|_{H^1(\Omega)} \|v_h\|_{L^2(\Omega)}.
\end{align*}
Here we used the Cauchy-Schwarz inequality in the second step, the finite overlap of the patches, and the norm equivalence of Corollary~\ref{cor:equiv} in the last. 
\end{proof}

\subsection{Discrete stability}

We now derive discrete stability estimates for solutions of Method~\ref{meth:main}. 
To simplify the presentation, we introduce the short-hand notation
\begin{align*}
&\|\v_h\|^2_{\alpha} \coloneqq \la\alpha \v_h,\v_h\ra,\qquad
\|\v_h\|^2_{h,\alpha} \coloneqq \la\alpha \v_h,\v_h\ra_h,
\end{align*}
for non-negative piecewise constant parameters $\alpha \in P_0(\Th)$. 
Furthermore, we use 
\begin{align}
\widehat u_h^{\,n+1/2} = \tfrac{1}{2}(u_h^n + u_h^{n+1}) 
\qquad \text{and} \qquad 
\dtau \widehat u_h^{\,n} = \tfrac{1}{2\tau}(u_h^{n+1} - u_h^{n-1}).
\end{align}
As a final ingredient, we introduce a discrete energy functional, defined by
\begin{align} \label{eq:energy}
\EE_h(u_h^n,u_h^{n+1}) &\coloneqq \|\dtau u_h^{n+1/2}\|_{h,\eps}^2 + \|\curl \wh u_h^{\,n+1/2}\|^2_{\nu} \\ 
&\qquad- \tfrac{\tau^2}{4}\|\curl\dtau u_h^{n+1/2}\|^2_{\nu} - \tfrac\tau2\left(\| \wPi_h \dtau u_h^{n+1/2}\|_{h,\sigma}^2-\|\dtau u_h^{n+1/2}\|_{h,\sigma}^2\right). \notag
\end{align} 
We begin with some elementary auxiliary observations. 
\begin{lemma}
Let (A1)--(A2) and (A4) hold and $u_h^n,u_h^{n+1} \in V_h$ be given. 
Then
\begin{align} \label{eq:equiv2}
\tfrac23\EE_h(u_h^n,u_h^{n+1}) \le \|\dtau u_h^{n+1/2}\|_{h,\eps}^2 + \|\curl \wh u_h^{\,n+1/2}\|^2_{\nu} \le 2\EE_h(u_h^n,u_h^{n+1}).
\end{align}
\end{lemma}

\begin{proof}

The result follows directly from the definition of the energy functional and the CFL condition (A4), which was specifically tailored to obtain this result.
\end{proof}

We can now establish the required stability estimate for solutions of Method~\ref{meth:main}. 
\def\a{\xi}
\begin{lemma}[Discrete stability] $ $\label{lem:discrete} \\
Let (A1)--(A2) and (A4) hold.
Further let $\a_h^n$, $\r_h^n \in V_h$, $n \ge 0$ be given such that
\begin{align} \label{eq:discreteeq}
\la(\eps + \tfrac\tau2\sigma) \dtautau \a_h^n,v_h\ra_h  + \la\sigma\wPi_h \dtau  \a_h^{n-1/2},\wPi_h v_h\ra_h + \la\nu\curl \a_h^n,\curl v_h\ra = \la r_h^n,v_h\ra, \qquad 
\end{align}
for all $v_h\in V_h$ and $n\ge 0$.
Then for all time steps $0 \le n <N$, there holds
\begin{align*}
\EE_h(\a_h^n,\a_h^{n+1}) \le \EE_h(\a_h^0,\a_h^1) + 2 \sum\nolimits_{k=1}^{n} \tau (\r_h^k,\dtau \wh\a_h^{\,k}).
\end{align*}
\end{lemma}
\begin{proof}
We begin by setting $v_h = \dtau \wh\a_h^{\,n}:=\frac{1}{2\tau}(\a_h^{n+1}-\a_h^{n-1})$ in the identity \eqref{eq:discreteeq}.
For the first term on the left-hand side, this results in
\begin{align}\label{eq:dproof1}
(\varepsilon \dtautau\a_h^n,\dtau \wh\a_h^{\,n})_h &= \frac{1}{2\tau}\big(\|\dtau\a_h^{n+1/2}\|_{h,\varepsilon}^2-\|\dtau\a_h^{n-1/2}\|_{h,\varepsilon}^2\big).
\end{align}
For the third term in \eqref{eq:discreteeq}, we see in a similar manner that
\begin{align}\label{eq:dproof2}
(\nu\curl\a_h^n,\curl \dtau \wh\a_h^{\,n})
&=\tfrac{1}{2\tau}\Big(\|\curl \wh\a_h^{\,n+1/2}\|^2_{\nu} -\|\curl \wh\a_h^{\,n-1/2}\|^2_{\nu} \\
&\qquad-\tfrac{\tau^2}{4}\|\curl\dtau\a_h^{n+1/2}\|^2_{\nu} + \tfrac{\tau^2}{4}\|\curl\dtau\a_h^{n-1/2}\|^2_{\nu} \Big). \nonumber
\end{align}
The loss terms involving $\sigma$ have to be treated more carefully. Here we use that
\begin{align} \label{eq:dproof3} 
\la\tfrac{\tau}{2}\sigma &\dtautau\a_h^n,\dtau \wh\a_h^{\,n}\ra_h + \la\sigma\wPi_h \dtau \a_h^{n-1/2}, \wPi_h \dtau \wh\a_h^{\,n}\ra_h \\
&= \la \tfrac\tau2\sigma \dtautau\a_h^n,\dtau \wh\a_h^{\,n}\ra_h - \la\tfrac\tau2\sigma \dtautau \wPi_h\a_h^n,\wPi_h \dtau  \wh\a_h^{\,n}\ra_h + \la\sigma\wPi_h \dtau \wh\a_h^{\,n},\wPi_h \dtau \wh\a_h^{\,n}\ra_h \notag \\
& = \tfrac\tau2 \tfrac1{2\tau}\Big(\|\dtau\a_h^{n+1/2}\|_{h,\sigma}^2 - \|\dtau\a_h^{n-1/2}\|_{h,\sigma}^2  \notag \\
& \qquad \qquad \qquad \qquad 
-\|\wPi_h \dtau \a_h^{n+1/2}\|_{h,\sigma}^2
 + \|\wPi_h \dtau \a_h^{n-1/2}\|_{h,\sigma}^2 \Big) + 
 \|\wPi_h \dtau \wh\a_h^{\,n}\|_{h,\sigma}^2. \notag
\end{align}
By summing up the identities \eqref{eq:dproof1}--\eqref{eq:dproof3}, we can deduce that 
\begin{align*}
\EE_h(\a_h^n,\a_h^{n+1}) &\le \EE_h(\a_h^n,\a_h^{n+1}) + 2\tau\|\wPi_h \dtau \wh\a_h^{\,n}\|_{\sigma}^2 
= \EE_h(\a_h^{n-1},\a_h^n) + 2\tau \la\r_h^n,\dtau\wh\a_h^{\,n}\ra.
\end{align*}
The assertion of the theorem then follows by induction over $n$.
\end{proof}

\subsection{Error estimates}

Let us start with proving the estimate of Theorem~\ref{thm:main1}. 
Following standard practice, we split
\begin{align}\label{eq:spliterror}
E(t^n)-E_h^n = -(\wPi_h E(t^n)-E(t^n)) + (\wPi_h E(t^n)-E_h^n)\eqqcolon -\eta^n+\xi_h^n,
\end{align}
into an interpolation error $\eta^n$ and a discrete error component $\xi_h^n$.
The use of the particular projection $\wPi_h$ in this splitting will become important below. 
With the interpolation error estimates of Lemma~\ref{lem:proj}, we immediately obtain
\begin{align*}
\max\limits_{0\le n< N}\Big(\|\dtau \eta^{n+1/2}&\|_{L^2}^2 + \|\curl\widehat \eta^{\,\,n+1/2}\|_{L^2}^2\Big) \\
&\le C h\left(\|\dt E\|_{L^\infty(H^1(\Th))} + \|\curl E\|_{L^\infty(H^1(\Th))}\right).
\end{align*}
For estimating the discrete error  $\xi_h^n$, we first note that $\xi_h^0=0$.
Moreover, $E_h^1=0$ by assumption, and by Taylor expansion, we further see that
\begin{align*}
E(\tau) &= E(0) + \tau \dt E(0) + \tfrac{\tau^2}{2} \dtt E(0) + \tfrac{\tau^3}{6} \dttt E(s_3),
\end{align*}
for some $0 < s_3 < \tau$.
From assumption (A3) and using equation~\eqref{eq:maxwell1}, one can see that $E(0)=\dt E(0)=\dtt E(0)=0$, and hence 
$\xi_h^1 = \frac{\tau^3}{6} \wt \Pi_h \dttt E(s_3)$.
Alternatively, we could also get $\xi_h^1=\frac{\tau^2}{2} \wt \Pi_h \dtt E(s_2)$ for some $0 < s_2 < \tau$ by truncating the Taylor series earlier.
In summary, this yields 
\begin{align}\label{eq:initvalest}
\begin{split}
\tfrac23\EE_h(\xi_h^0,\xi_h^{1}) &\le \|\dtau \xi_h^{1/2}\|_{L^2}^2 + \|\curl\widehat \xi_h^{\;1/2}\|_{L^2}^2 
= \tfrac1{\tau^2}\|\xi_h^1\|_{L^2}^2 + \tfrac14\|\curl\xi_h^1\|_{L^2}^2\\
&\le C \tau^4 \Big(\|\dttt E\|_{L^\infty(H^1(\Th))}^2+ \|\dtt E\|_{L^\infty(H^1(\Th))}^2\Big).
\end{split}
\end{align}
For the last step, we used the formulas for $\xi_h^0$ and $\xi_h^1$ derived above, and the stability estimates $\|\wt \Pi_h v\|_{L^2} \le C \|v\|_{H^1}$ and $\|\curl \wt \Pi_h v\|_{L^2} \le C \|v\|_{H^1}$ for the projection. 

As a next step, let us observe that the continuous solution $E(t)$ of \eqref{eq:maxwell1}--\eqref{eq:maxwell2} satisfies
\begin{align*}
\la\varepsilon \dtt\E(t^n),v_h\ra + \la\sigma\dt\E(t^n), v_h\ra &+ \la\nu\curl\E(t^n),\curl v_h\ra
= \la f(t^n),v_h\ra + \la g(t^n),v_h\ra_{\partial\Omega}
\end{align*}
for all $v_h\in V_h$ and $n>0$.
By combination with \eqref{eq:method}, one can then see that the discrete error $\xi_h^n = \wt\Pi_h E(t^n) - E_h^n$ thus satisfies the error equation \eqref{eq:discreteeq} with 
\begin{align*}
\la\r_h^n,v_h\ra = \la\r_{h,t}^n,v_h\ra + \la \r_{h,s}^n,v_h\ra + \la\r_{h,q}^n,v_h\ra + \la\r_{h,p}^n,v_h\ra,
\end{align*}
where the four partial residuals are defined by 
\begin{align*}
\la\r_{h,t}^n,v_h\ra &= \la\varepsilon(\dtautau E^n -\dtt E^n),v_h\ra + \la\sigma(\dtau \wh E^n-\dt E^n),v_h\ra, \\
\la\r_{h,s}^n,v_h\ra &= \la\varepsilon\dtt\eta^n,v_h\ra + \la\sigma\dt\eta^n,v_h\ra+ \la\nu\curl\eta^n,\curl v_h\ra, \\
\la\r_{h,q}^n,v_h\ra &= \delta_h((\eps+\tfrac\tau2\sigma) \wPi_h \dtautau E^n, v_h) + 
\delta_h(\sigma \wPi_h \dtau E^{n-1/2}, v_h), \\
\la\r_{h,p}^n,v_h\ra &= \la\wPi_h \sigma \dtau E^{\,n-1/2},\wPi_h v_h- v_h\ra,
\end{align*}
which represent the temporal, spatial, quadrature, and projection errors, respectively.
In the third term, we again used $\delta_h(\alpha u_h,v_h) = \la\alpha u_h,v_h\ra_h - \la\alpha u_h,v_h\ra$ to abbreviate the quadrature error.
We now estimate the four residuals independently.

\medskip 
\noindent 
\textit{First residual.}
By summation over the time steps, we get
\begin{align*}
\sum\nolimits_k \tau\la\r_{h,t}^k,\dtau\wh\xi_h^{\,k}\ra
= \sum\nolimits_k  \tau&\la\varepsilon(\dtautau E^k -\dtt E^k),\dtau\wh\xi_h^{\,k}\ra \\
&+ \sum\nolimits_k \tau\la\sigma(\dtau \wh E^k-\dt E^k),\dtau\wh\xi_h^{\,k}\ra
=(i) + (ii).
\end{align*}
The first term in this expansion can be estimated by Taylor expansion, giving
\begin{align*}
|(i)| &\le \sum\nolimits_k c \tau^2 \|\dtttt E\|_{L^1(t^{k-1},t^k;L^2)}\|\dtau \widehat\xi_h^{\,k}\|_{L^2}\\
&\leq C\tau^{4}\|\dtttt E\|_{L^1(L^2)}^2+\tfrac{1}{56}\|\dtau \xi_h\|_{\ell_\infty(L^2)}^2 \\
&\leq C\tau^{4}\|\dtttt E\|_{L^1(L^2)}^2+\tfrac{1}{28}\max_{0\le k<N} \EE_h(\xi_h^k,\xi_h^{k+1}).
\end{align*}
With similar arguments, the second term can be bounded by
\begin{align*}
|(ii)| &\leq C \tau^{4}\|\dttt E\|_{L^1(L^2)}^2+\tfrac{1}{28}\max_{0\le k<N} \EE_h(\xi_h^k,\xi_h^{k+1})
\end{align*}

\medskip 
\noindent 
\textit{Second residual.}
For the spatial errors, we use
\begin{align*}
\sum\nolimits_k\tau\la\r_{h,s}^k,\dtau\wh\xi_h^{\,k}\ra
&= \sum\nolimits_k \tau\la\varepsilon\dtt \eta^k,\dtau\wh\xi_h^{\,k}\ra + \sum\nolimits_k \tau \la\sigma\dt\eta^k,\dtau\wh\xi_h^{\,k}\ra  \\
&\qquad + \sum\nolimits_k \tau\la\nu\curl\eta^k,\curl\dtau\wh\xi_h^{\,k}\ra = (iii) + (iv) + (v).
\end{align*}
By the interpolation error estimates of Lemma~\ref{lem:proj} and Young's inequality, we obtain
\begin{align*}
|(iii)+(iv)| 
&\le C h^2 (\|\dt E\|^2_{L^\infty(H^1(\Th))} + \|\dtt E\|^2_{L^\infty(H^1(\Th))}) 
+\tfrac{1}{28}\max_{0\le k<N} \EE_h(\xi_h^k,\xi_h^{k+1}).
\end{align*}
The factor $\tfrac{1}{28}$ in the last term was simply chosen small enough for later on. 
For the remaining term in the second residual, we use summation by parts to arrive at 
\begin{align*}
|(v)| \le \la\curl&\eta^n,\curl\wh\xi_h^{\;n+1/2}\ra -\la\curl\eta^1,\curl\wh\xi_h^{\;1/2}\ra \\
&-\sum\nolimits_k \la\nu\curl\dtau\eta^{k-1/2},\curl\wh\xi_h^{\;k-1/2}\ra
=(vi) + (vii) + (viii).
\end{align*}
By the interpolation error estimates and Young's inequality, we obtain 
\begin{align*}
|(vi) + (vii)|
&\le C h^2 \|\curl\E\|^2_{L^\infty(H^1(\Th))} + \tfrac{1}{56}\|\curl\wh\xi_h\|^2_{\ell_\infty(L^2)}\\
&\le C h^2 \|\curl\E\|^2_{L^\infty(H^1(\Th))} + \tfrac{1}{28}\max_{0\le k<N} \EE_h(\xi_h^k,\xi_h^{k+1}).
\end{align*}
By Taylor expansion and arguments similar to before, we further obtain
\begin{align*}
|(viii)|&\le C h^{2} \|\curl\dt E\|_{L^1(H^1(\Th))}^2 + \tfrac{1}{28}\max_{0\le k<N} \EE_h(\xi_h^k,\xi_h^{k+1}).
\end{align*}
The prefactors in these estimates were again simply chosen sufficiently small.

\medskip 
\noindent 
\textit{Third residual.}
Using Lemma~\ref{lem:quaderror} and Taylor expansion in time, we can estimate the residuals caused by the quadrature errors via
\begin{align*}
\sum\nolimits_k\tau\la\r_{h,q}^k,\dtau\wh\xi_h^{\,k}\ra
&\le  C h^2 (\|\dtt E\|^2_{L^1(H^1(\Th))} + \|\dt E\|^2_{L^1(H^1(\Th))}) + \tfrac{1}{56} \|\dtau\xi_h\|^2_{\ell_\infty(L^2)}\\
&\le  C h^2 (\|\dtt E\|^2_{L^1(H^1(\Th))} + \|\dt E\|^2_{L^1(H^1(\Th))}) + \tfrac{1}{28} \max_{0\le k<N} \EE_h(\xi_h^k,\xi_h^{k+1}).
\end{align*}

\medskip 
\noindent 
\textit{Fourth residual.}
For the projection errors, we again use Lemma~\ref{lem:projprop} and obtain
\begin{align*}
\sum\nolimits_k\tau\la\r_{h,p}^k,\dtau\wh\xi_h^{\,k}\ra
&= \sum\nolimits_k \tau \la\sigma \wPi_h \dtau E^{\,k-1/2},\wPi_h\dtau\wh\xi_h^{\,k}-\dtau\wh\xi_h^{\,k}\ra \\
&\le C h^2\|\dt E\|_{L^1(H^1)}^2 + \tfrac{1}{56}\|\dtau\xi_h\|^2_{\ell_\infty(L^2)} \\
&\le C h^2\|\dt E\|_{L^1(H^1)}^2 + \tfrac{1}{28}\max_{0\le k<N} \EE_h(\xi_h^k,\xi_h^{k+1}).
\end{align*}
Let us note that the appearance of the projection $\wPi_h$ in the residual $r_{h,p}^k$ was essential here, in order to be able to apply Lemma~\ref{lem:projprop}.

\medskip 
\noindent 
\textit{Conclusion.}
By summation of all the individual estimates for the respective residuals and application of Lemma~\ref{lem:discrete}, we finally obtain
\begin{align*}
\EE(\xi_h^n,\xi_h^{n+1})&\le \EE(\xi_h^0,\xi_h^{1}) + 2\left(7\cdot \tfrac{1}{28} \max_{0\le k < N} \EE(\xi_h^k,\xi_h^{k+1}) + C(E) h^2 + C'(E) \tau^2\right).
\end{align*}
Taking the maximum over all $n$, using \eqref{eq:equiv2}, and the bound \eqref{eq:initvalest} for $\EE(\xi_h^0,\xi_h^{1})$, we thus obtain the required estimates for the discrete error component. By combination with the interpolation error estimate, we arrive at the estimate of Theorem~\ref{thm:main1}.
\qed

\subsection*{Estimate of Theorem~\ref{thm:main2}}
We now split the corresponding error by 
\begin{align*}
E(t^n) - \wPi_h E_h^n = -(\wPi_h E(t^n) - E(t^n)) + (\wPi_h E(t^n) - \wPi_h E_h^n) =: -\eta^n + \wt \xi_h^n
\end{align*}
The first term is the same as before, and for the discrete error, we can use the following arguments: By the commuting diagram property of the projectors, we see that $\curl \wt \xi_h^n = \curl \xi_h^n$, which allows using the bounds of the previous proof to handle the $\curl$-terms in the estimate of the discrete error $\wt \xi_h^n$.
For the $L^2$-terms, we use 
\begin{align*}
\|\dtau \wt \xi_h^{n+1/2}\|_{L^2} 
&= \|\wPi_h(\Pi_h \dtau E^{n+1/2}-\dtau  E_h^{n+1/2})\|_{L^2}
\le C \|\Pi_h \dtau E^{n+1/2} - \dtau E_h^{n+1/2}\|_{L^2} \\
&\le C (\|\wPi \dtau E^{n+1/2} - \dtau E_h^{n+1/2}\|_{L^2} + \|\Pi_h \dtau E^{n+1/2} - \wPi_h \dtau E^{n+1/2}\|_{L^2}).
\end{align*}
In the first inequality, we made use of the bound $\|\wt \Pi_h v_h\|_{L^2} \le C \|v_h\|_{L^2}$, which follows by the usual scaling arguments. 
The first term in the above estimate only involves the discrete error $\dtau \xi_h^{n+1/2} = \wPi \dtau E^{n+1/2} - \dtau E_h^{n+1/2}$, which was already analyzed before, and the second term can be bounded by the interpolation error estimates for the two projectors.
This then already yields the estimate of Theorem~\ref{thm:main2}.
\qed

\section{Algebraic properties}\label{sec:implementation}

We now discuss the implementation of the proposed methods and the algebraic properties stated in Theorems~\ref{thm:main1} and \ref{thm:main2}.
Let us recall the basis functions \begin{align*} %
\Phi_{ij} = \lambda_i \nabla \lambda_j \qquad \text{and} \qquad \Phi_{ji} = -\lambda_j \nabla \lambda_i
\end{align*}
associated with the edges $e_{ij} \in \calE_h$. By basic computations, one can see the following. 
\begin{lemma} \label{lem:alg_proj}
Let $v_h = \sum_{e_{ij}} \ttv_{ij} \Phi_{ij} + \ttv_{ji} \Phi_{ji} \in V_h$ be given. 
Then 
\begin{align}
\wt\Pi_h v_h = \sum_{e_{ij} \in \calE_h \setminus \wt \calE_h} \ttv_{ij} \Phi_{ij} + \ttv_{ji} \Phi_{ji} + \sum_{e_{ij} \in \wt \calE_h} \tfrac{1}{2}(\ttv_{ij} + \ttv_{ji}) (\Phi_{ij} + \Phi_{ji}).
\end{align}
Hence $\wt \Pi_h v_h = \sum_{e_{ij} \in \calE_h} \wh \ttv_{ij} \Phi_{ij} + \wh \ttv_{ji} \Phi_{ji}$ with coefficients given by $\wh \ttv_{ij}\coloneqq\ttv_{ij}$, $\wh \ttv_{ji}\coloneqq\ttv_{ji}$ for the edges $e_{ij} \in \calE_h \setminus\wt\calE_h$ and $\wh\ttv_{ij}=\wh\ttv_{ji}\coloneqq\frac{1}{2}(\ttv_{ij}+\ttv_{ji})$ for the edges $e_{ij}\in\wt\calE_h$.
\end{lemma}
For the following considerations, we always assume that the degrees of freedom are sorted edge-wise. 
This allows us to make the following statements.
\begin{lemma}\label{lem:alg1}
The equation \eqref{eq:method} is equivalent to the algebraic system \eqref{eq:impl} with 
\begin{alignat*}{2}
[\ttM_{\eps+\tau \sigma/2}]_{ij,kl} &=\la(\eps + \tau \sigma/2) \Phi_{kl},\Phi_{ij}\ra_h, \qquad & 
[\ttf^n]_{ij} &= \la f(t^n),\Phi_{ij}\ra, \\
[\ttK_\nu]_{ij,kl} &= \la\nu \curl \Phi_{kl}, \curl \Phi_{ij}\ra, \qquad &
[\ttg^n]_{ij} &= \la g(t^n),\Phi_{ij}\ra_{\partial\Omega}.
\end{alignat*}
Both matrices are sparse, and $\ttM_{\eps+\tau \sigma/2}$ is block diagonal, with one block per vertex of the mesh coupling all degrees of freedom not vanishing at this vertex. 
Furthermore 
\begin{align*}
\widehat \ttM_\sigma = \ttQ^\top \ttM_\sigma \ttQ 
\qquad \text{with} \qquad  
[\ttM_\sigma]_{ij,kl} = (\sigma \Phi_{kl}, \Phi_{ij})_h,
\end{align*}
and the projection matrix $\ttQ$ is block diagonal with $2 \times 2$ blocks of the form
\begin{align*}
\frac{1}{2}\begin{pmatrix}1 & 1 \\ 1 & 1\end{pmatrix}
\qquad \text{and} \qquad 
\begin{pmatrix}1 & 0 \\ 0 & 1\end{pmatrix}
\end{align*}
for the edges $e_{ij} \in \wt \calE_h$ and $e_{ij} \in \calE_h \setminus \wt \calE_h$, respectively. 
\end{lemma}
The assertions follow immediately from the properties of the basis functions; for details, see \cite{EggerRadu20c_maxwellyee,Radu22}.
This already proves the second claim of Theorem~\ref{thm:main1}. \qed

\subsection*{Algebraic reduction}
We next derive the reduced scheme \eqref{eq:impl2}. To do so, we start with the following observation, which again follows from elementary arguments.
\begin{lemma}
The function $\wt\Pi_h v_h$ can be expressed equivalently as 
\begin{align*}
\wt\Pi_h v_h = \sum_{e_{ij} \in \calE_h \setminus \wt\calE_h} \wt \ttv_{ij} \Phi_{ij} + \wt \ttv_{ji} \Phi_{ji} + \sum_{e_{ij} \in \wt\calE_h} \wt\ttv_{ij} \wt \Phi_{ij}, \qquad \wt \Phi_{ij} = \Phi_{ij} + \Phi_{ji},
\end{align*}
with $\wt \ttv_{ij}=\ttv_{ij}$, $\wt \ttv_{ji}=\ttv_{ji}$ for $e_{ij} \in \calE_h \setminus\wt\calE_h$ and $\wt\ttv_{ij}=\frac{1}{2}(\ttv_{ij}+\ttv_{ji})$ for $e_{ij}\in\wt\calE_h$, and the collection $\{\Phi_{ij},\Phi_{ji} : e_{ij} \in \calE_h \setminus \wt \calE_h\} \cup \{ \wt \Phi_{ij}:=\Phi_{ij} + \Phi_{ji}) : e_{ij} \in \wt \calE_h\}$ is a basis for $\wt V_h$.
\end{lemma}
Also for the reduced space, the degrees of freedom are sorted edgewise. 
Then the relation between the coefficients $\ttv_{ij}$, $\wh \ttv_{ij}$ and $\wt \ttv_{ij}$ can be expressed as follows.
\begin{lemma} \label{lem:prolongation}
Let $\ttP$ be the block-diagonal (prolongation) matrix with blocks 
\begin{align} \label{eq:prolongation}
\begin{pmatrix} 1 \\ 1 \end{pmatrix}
\qquad \text{and} \qquad 
\begin{pmatrix}
1 & 0 \\ 0 & 1
\end{pmatrix}
\end{align}
for edges $e_{ij} \in \wt\calE_h$ and $e_{ij} \in \calE_h \setminus \wt \calE_h$, respectively. 
Then $\ttQ=\ttP (\ttP^\top \ttP)^{-1} \ttP^\top$ and $\ttP^\top \ttP$ is diagonal with entries $1/2$ and $1$, respectively.
Furthermore, 
\begin{align} \label{eq:restriction}
\wh \ttv = \ttP\,\wt \ttv \qquad \text{and} \qquad \wt \ttv = \ttR\,\ttv \quad \text{with} \quad 
\ttR=(\ttP^\top \ttP)^{-1} \ttP^\top.
\end{align}
\end{lemma}
This lemma allows us to express the coefficients of $\wt\Pi_h v_h$ in the basis of $\wt V_h$ by the coefficients of the expansion in the basis of $V_h$. 
With the help of this result, we can now derive the algebraic form \eqref{eq:impl2} of the reduced scheme of Theorem~\ref{thm:main2}.
\begin{lemma}\label{lem:alg2}
Let $(\ttE^n)_n$ be a solution of \eqref{eq:impl}. Then $\wt \ttE^n = \ttR \ttE^n$ satisfies \eqref{eq:impl2} with 
\begin{align*}
\wt \ttM_{\eps + \tau \sigma/2}^{-1} &= \ttR \ttM_{\eps+\tau \sigma/2}^{-1} \ttR^\top,  \qquad 
\wt \ttM_\sigma=\ttP^\top\ttM_\sigma \ttP,  \qquad 
\wt \ttK_\nu = \ttP^\top \ttK_\nu \ttP
\end{align*}
and right-hand sides
\begin{align}\label{eq:rhs1}
\wt \ttF^n = \ttR\ttM_{\eps+\tau \sigma/2}^{-1}\ttf^n  \qquad 
\wt \ttG^n = \ttR\ttM_{\eps+\tau \sigma/2}^{-1}\ttg^n,
\end{align}
where $\ttf^n,\ttg^n$ are defined as in Lemma~\ref{lem:alg1}.
Hence, the projection $\wt E_h^n=\wt \Pi_h E_h^n$ of the solution of Method~\ref{meth:main} is given by \eqref{eq:impl2}. 
\end{lemma}
\begin{proof}
Since $\ttM_{\eps + \tau \sigma/2}$ is regular, the scheme \eqref{eq:impl} is equivalent to 
\begin{align*}
\dtautau \ttE^n &= \ttM_{\eps + \tau \sigma/2}^{-1} \left( -\wh \ttM_\sigma \dtau \ttE^{n-1/2} - \ttK_\nu \ttE^n + \ttf^n + \ttg^n\right). 
\end{align*}
We further multiply this equation from the left by $\ttR^\top = \ttP (\ttP^\top \ttP)^{-1}$ and note that $\ttP \ttR = \ttP (\ttP^\top \ttP)^{-1} \ttP^\top = \ttQ$, which follows from the definition of the matrices. 
We see
\begin{align*}
\wh \ttM_\sigma = \ttR^\top \ttP^\top  \ttM_\sigma \ttP \ttR 
\qquad \text{and} \qquad 
\ttK_\nu = \ttQ^\top \ttK_\nu \ttQ 
= \ttR^\top \wt \ttK_\nu \ttR, 
\end{align*}
where we used the algebraic form of the commuting diagram property of the projection $\wt \Pi_h$, see \eqref{eq:curlcommuting}.
Using the definitions of $\wt \ttE^n$, $\wt \ttF^n$, and $\wt \ttG^n$, we thus conclude that
\begin{align*}
\dtautau \wt \ttE^n 
= \dtautau \ttR \ttE^n
&= \ttR \ttM_{\eps + \tau \sigma/2}^{-1} (-\ttR^\top \wt \ttM_\sigma \ttR \,  \dtau \ttE^{n-1/2} - \ttR^\top \wt \ttK_\nu \ttR \ttE^n + \ttf^n + \ttg^n) \\
&= \wt \ttM_{\eps + \tau \sigma/2}^{-1} (- \wt \ttM_\sigma \dtau \wt \ttE^{n-1/2} - \wt \ttK_\nu \wt \ttE^n) + \wt \ttF^n + \wt \ttG^n
\end{align*}
This already yields the algebraic form given in \eqref{eq:impl2} and concludes the proof.
\end{proof}

\begin{remark}\label{rem:rhs}
The choice \eqref{eq:rhs1} for the right-hand sides $\wt \ttF^n$, $\wt \ttG^n$ in the reduced method makes use of the vectors $\ttf^n$, $\ttg^n$ with two degrees of freedom for each edge. 
An alternative and more direct choice would be
\begin{align}\label{eq:rhs2}
\wt \ttF^n := \wt \ttM_{\eps + \tau \sigma/2}^{-1}\wt\ttf^n  \qquad \text{and}\qquad
\wt \ttG^n := \wt \ttM_{\eps + \tau \sigma/2}^{-1}\wt\ttg^n,
\end{align}
which now only depends on the vectors $\wt\ttf^n$ and $\wt\ttg^n$ assembled on the reduced finite element space $\wt V_h$. 
This modification can be included in our analysis, if one replaces assumption (A5) by choosing $\wt\calE_h \subset \calE_h$ such that
\begin{enumerate}[topsep=1em]
\item[(A5$^*$)] \label{ass:A5m}
$\sigma$ and $f$ are continuous across edges $e \in \wt \calE_h$ lying in the interior $\Omega$, and $\sigma=0$, $g=0$ for all edges $e \in \wt \calE_h$ on the boundary $\partial\Omega$.
\end{enumerate}
This allows us to apply a variant of Lemma~\ref{lem:projprop} which covers the additional consistency error introduced by the representation of $f$; the boundary term $g$ is already fully covered by the original assumption (A5) since it vanishes along all edges where no reduction is applied.
The time-stepping scheme resulting from this modified choice of the right-hand sides can be stated as 
\begin{align} \label{eq:impl3}
\dtautau \wt \ttE^n 
= \wt \ttM_{\eps + \tau \sigma/2}^{-1} (- \wt \ttM_\sigma \dtau \wt \ttE^{n-1/2} - \wt \ttK_\nu \wt \ttE^n + \wt\ttf^n + \wt\ttg^n),
\end{align}
and used as an alternative to \eqref{eq:impl2}. 
In our numerical tests, we will compare the two versions \eqref{eq:impl2} and \eqref{eq:impl3}, in particular, also highlighting the importance of the additional conditions in assumption (A5$^*$) for the latter.
\end{remark}

\section{Numerical validation}\label{sec:numerics}
As a model problem for our numerical tests, we consider the scattering of a plane electromagnetic wave from a cylinder. 
Under the usual symmetry assumptions, this can be modeled by Maxwell's equations in two space dimensions. 
All results from the three-dimensional setting translate almost verbatim. 
The main differences are that  $\Th$ now corresponds to a triangular mesh and that the space $\wt V_h$ is a subspace of $P_1(\K)^2\cap H(\curl;\Omega)$. 
Moreover, the triangular vertex rule
\begin{align}
(a,b)_h = \sum\nolimits_T   \tfrac{|T|}{3} \sum\nolimits_{v_i \in T} a(v_i) \cdot b(v_i)
\end{align}
is used for the numerical quadrature.

\subsection*{Model problem.}
The geometric setup of our test problem is illustrated in Figure~\ref{fig:2}. 
The computational domain $\Omega=(-1,1)^2$ is split into two parts: 
The circular region $\Omega_S$ of radius $r_S=0.3$ contains the scatterer with material parameters $\varepsilon=\nu=1$ and $\sigma=100$. 
For the surrounding medium $\Omega \setminus \Omega_S$ we take $\varepsilon=\nu=1$ and $\sigma=0$.
\begin{figure}[ht!]
\centering
\begin{tikzpicture}[scale=2]
\draw[black, line width=0.0mm, fill=blue, opacity = 0] (-1,-1) -- (-1, 1) -- ( 1, 1) -- ( 1,-1) -- cycle;
\draw[black, line width=0.5mm] (-1,-1) -- (-1, 1) -- ( 1, 1) -- ( 1,-1) -- cycle;
\draw[black, line width=0.0mm, fill=red, opacity = 0] {(0,0) circle (0.33)}; %
\draw[black, line width=0.5mm] {(0,0) circle (0.33)};
\node[text width=3cm] at (0.64,0) {$\Omega_S$};
\node[text width=3cm] at (1,-0.64) {$\Omega_{R\setminus S}$};
\draw[] (-1,-1) node[below] {\tiny $(-1,-1)$};
\draw[] (1,-1) node[below] {\tiny $(1,-1)$};
\draw[] (-1,1) node[above] {\tiny $(-1,1)$};
\draw[] (1,1) node[above] {\tiny $(1,1)$};
\end{tikzpicture}
\caption{Computational domain for the wave scattering problem.\label{fig:2}}
\end{figure}
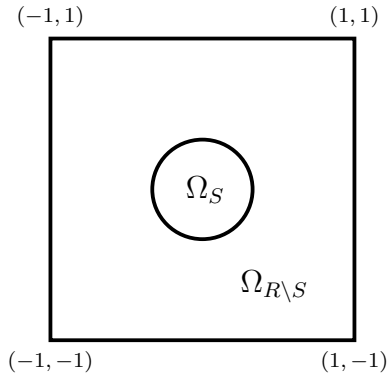

\subsection*{Test scenario}
For excitation of the fields, the initial and boundary conditions are chosen to match the corresponding traces of the plane wave 
\begin{align*}
\widetilde E(x,y,t)=
\begin{pmatrix}
-k_2 \\ k_1    
\end{pmatrix}
a(k_1x+k_2x-t)
\end{align*}
with envelope $a(x)=2e^{-10(x+3)^2}$ and $k=(k_1,k_2)=\frac1{\sqrt{2}}(1,1)$. 
The function $\widetilde E$ 
satisfies \eqref{eq:maxwell1}--\eqref{eq:maxwell2} with $\varepsilon=\nu=1$, $\sigma=0$, and data $f=0$ and $g=n \times (\nu \curl \widetilde E)$.
As a consequence, the plane wave will first propagate freely through $\Omega \setminus \Omega_S$, but then get scattered at the inclusion $\Omega_S$ with high conductivity. 
The scattered wave will then travel back through the free region $\Omega \setminus \Omega_S$ and finally be artificially reflected at the outer boundary $\partial\Omega$.
In Figure~\ref{fig:wave}, some corresponding snapshots of the numerical solution $E_h(t)$ are depicted. 
\begin{figure}[ht!]
    \centering
    \includegraphics[width=0.31\textwidth]{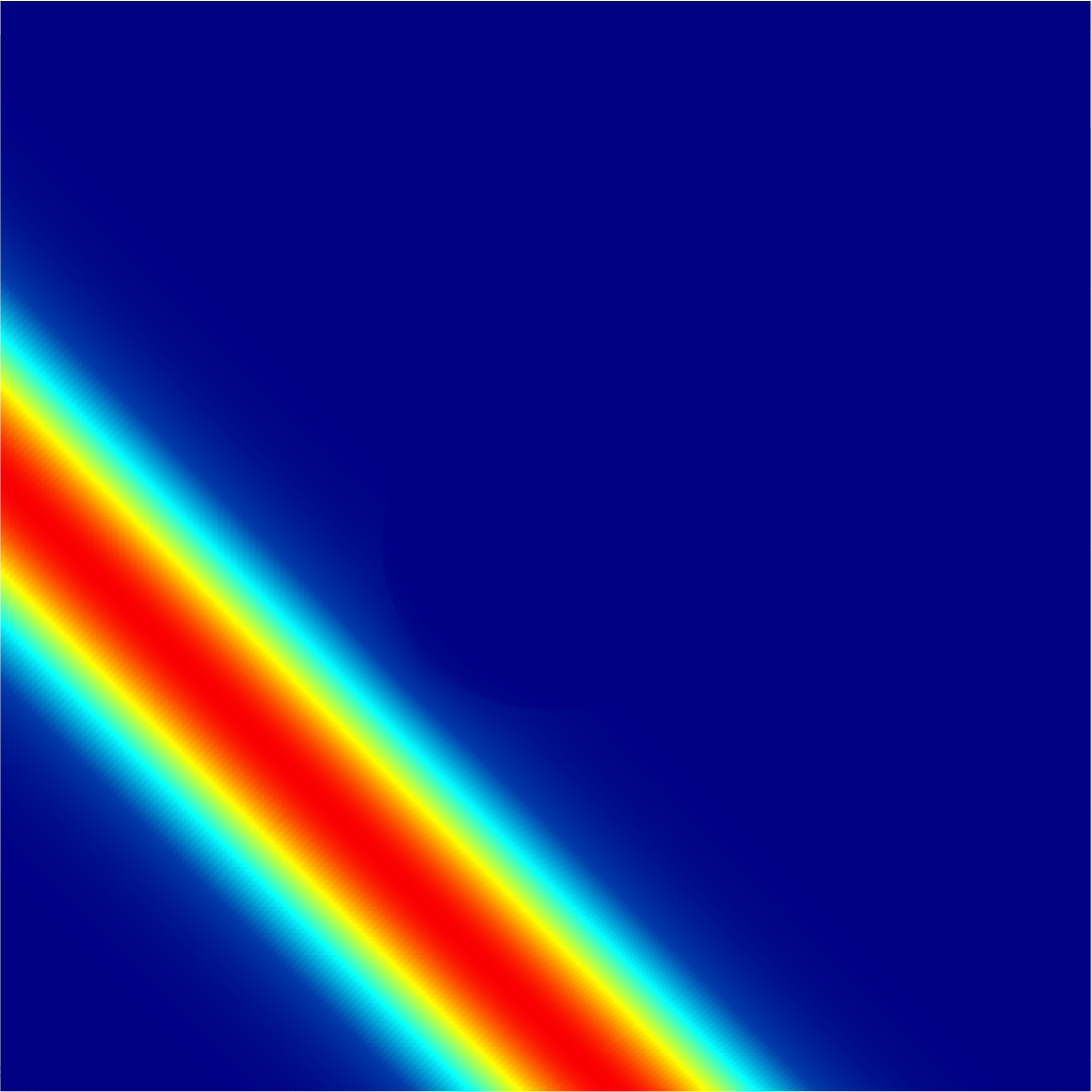}\quad
    \includegraphics[width=0.31\textwidth]{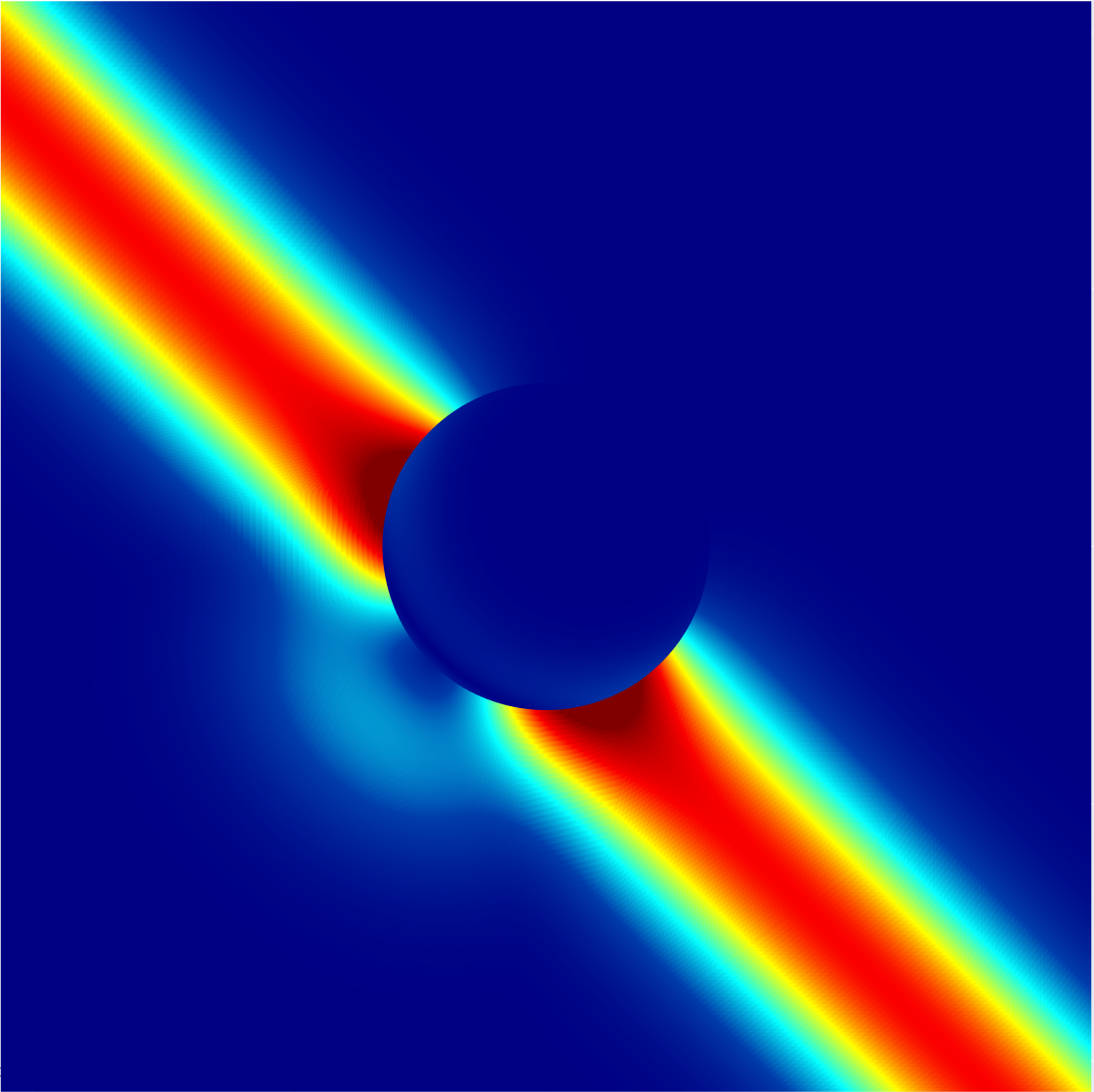}\quad
    \includegraphics[width=0.31\textwidth]{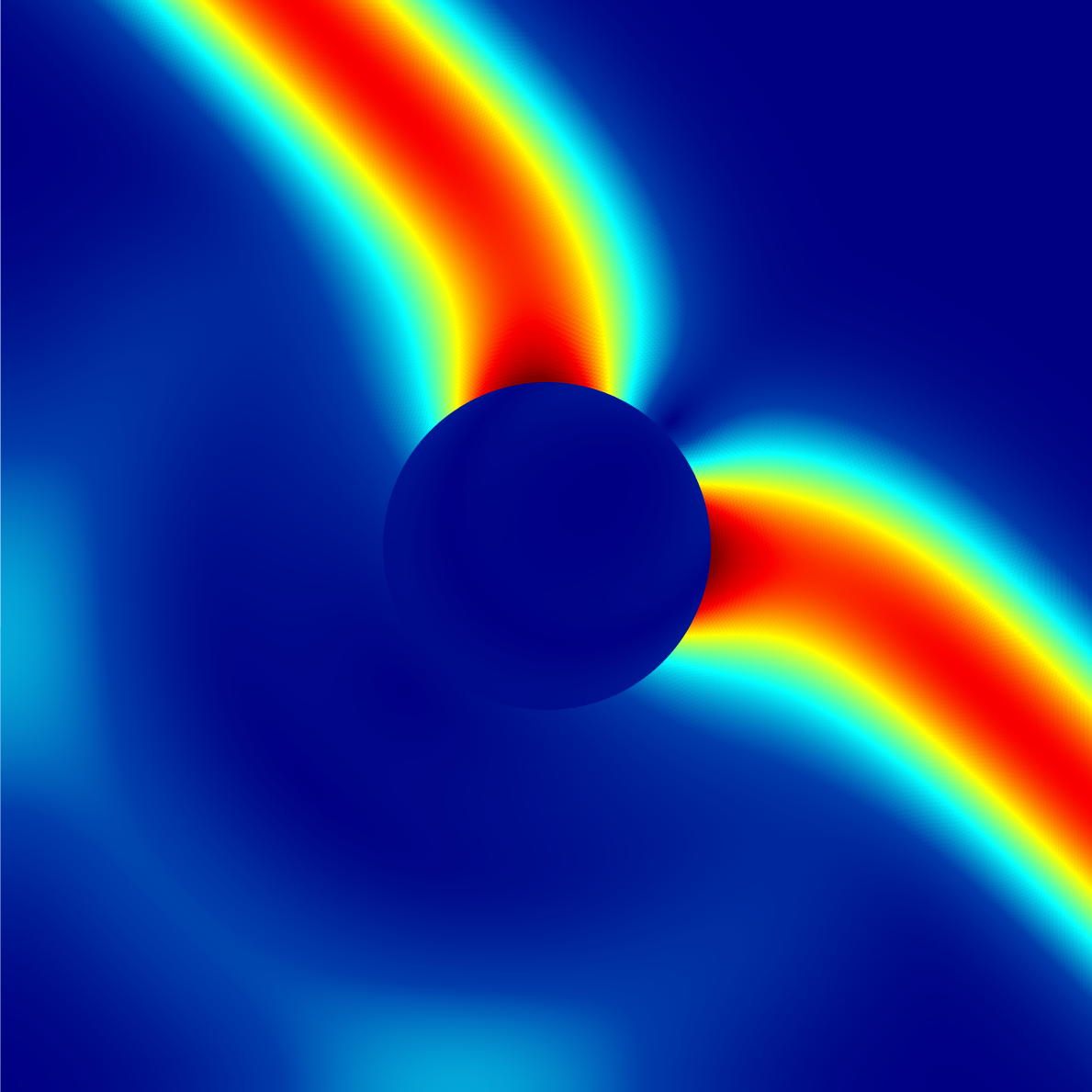}
    \caption{Absolute values of the simulated electric field $E_h(t)$ at time $t=1.5$ (left), $t=2$ (middle), and $t=2.5$ (right).}
    \label{fig:wave}
\end{figure}

\subsection*{Discretization setup}

To examine the convergence of the methods, we consider a sequence of meshes $\Th$ obtained by uniform refinement of an initial unstructured coarse mesh. The time step $\tau$ is chosen following the CFL condition (A4).
Let us note that $\sigma$ is continuous everywhere but on the edges along $\partial\Omega_S$ and vanishes on the outer boundary $\partial\Omega$.
To illustrate our theoretical results, we consider Method~\ref{meth:main} for three different choices of $\wt\calE_h$, namely 
\begin{itemize}
\item $\wt\calE_h = \emptyset$: This leads to $\wt V_h=\NC_1$ with two degrees of freedom for every edge;
\item $\wt\calE_h\subseteq \calE_h$ such that (A5) is satisfied, i.e. $\N_0 \subset \wt V_h \subset \NC_1$: in particular, only one degree of freedom is used for all edges not lying on the boundary $\partial \Omega_S$;  
\item $\wt\calE_h=\calE_h$: This amounts to $\wt V_h = \N_0$ and violates assumption (A5).
\end{itemize}
The corresponding finite element solutions are denoted by $E_h^\NC$, $E_h^{\N^+}$ and $E_h^\N$.
For measuring the error between two functions $E_h$ and $E_h^*$, we use 
\begin{align*}
\tnorm E_h - E_h^* \tnorm \coloneqq \frac{\|\dtau (E_h - E_h^*)\|_{\ell_\infty(L^2(\Omega))}}{\|\dtau E_h^*\|_{\ell_\infty(L^2(\Omega))}} &+ \frac{\|\curl (\widehat E_h - \widehat E_h^*)\|_{\ell_\infty(L^2(\Omega))}}{\|\curl \widehat E_h^*\|_{\ell_\infty(L^2(\Omega))}}
\end{align*}

\subsection*{Computational results}
We first study the error and convergence rate of the $\NC_1$ method.
Since the exact solution to our model problem is somewhat cumbersome to compute, we estimate the discretization errors by comparing the numerical solutions obtained on two different nested meshes. 
\begin{table}[ht!]
\begin{tabular}{c||c|c|c} 
$h\approx$ & $\tnorm E^\NC_{h}-E^\NC_{2h} \tnorm$  & eoc & dofs \\
\hline
\hline
\rule{0pt}{2.1ex}
$2^{-3}$ & $0.779603$ & ---    & 2.296 \\
$2^{-4}$ & $0.373586$ & $1.06$ & 9.056 \\
$2^{-5}$ & $0.185312$ & $1.01$ & 35.968 \\  
$2^{-6}$ & $0.093906$ & $0.98$ & 143.360 \\  
$2^{-7}$ & $0.046257$ & $1.02$ & 572.416 \\  
$2^{-8}$ & $0.023158$ & $1.00$ & 2.287.616 \\  
$2^{-9}$ & $0.011612$ & $1.00$ & 9.146.368 \\
\end{tabular}
\caption{Errors, the estimated order of convergence (eoc) and the number of degrees of freedom (dofs) for a multitude of mesh sizes $h$ with fixed time step sizes $\tau = 0.28\,h$.\label{tab:tab1}}
\end{table}
As one can infer from Table~\ref{tab:tab1}, first-order convergence of the $\NC_1$-method is observed, which is in perfect agreement with our previous studies~\cite{EggerRadu20c_maxwellyee, Radu22}.

In the following tests, we evaluate the convergence rates for the $\N_0^+$ and $\N_0$ methods by comparing them to the simulations obtained with the $\NC_1$ method. 
\begin{table}[ht!]
\centering
\begin{tabular}{c||c|c|c||c|c|c} 
$h\approx$ & $\tnorm E_h^{\N^+}-E_h^\NC\tnorm$ & eoc & dofs & $\tnorm E_h^{\N}-E_h^\NC \tnorm $  & eoc & dofs \\
\hline
\hline
\rule{0pt}{2.1ex}
$2^{-3}$ & $0.098385$ & ---    & 1.164     & $0.105887$ & ---    & 1.148 \\
$2^{-4}$ & $0.049076$ & $1.00$ & 4.560     & $0.053302$ & $1.00$ & 4.528 \\
$2^{-5}$ & $0.024133$ & $1.02$ & 18.048    & $0.027561$ & $0.95$ & 17.984 \\
$2^{-6}$ & $0.011923$ & $1.01$ & 71.808    & $0.015084$ & $0.87$ & 71.680 \\
$2^{-7}$ & $0.005916$ & $1.01$ & 286.464   & $0.008739$ & $0.79$ & 286.208 \\
$2^{-8}$ & $0.002945$ & $1.00$ & 1.144.320 & $0.005371$ & $0.70$ & 1.143.808 \\
$2^{-9}$ & $0.001469$ & $1.00$ & 4.574.208 & $0.003474$ & $0.63$ & 4.573.184
\end{tabular}
\medskip
\caption{Errors, the estimated order of convergence (eoc) and the number of degrees of freedom (dofs) for a multitude of mesh sizes $h$ with fixed time step sizes $\tau = 0.28\,h$.\label{tab:tab2}} 
\end{table}
For the $\N_0^+$-method, assumption (A5) is valid, and we observe first-order convergence as predicted by our theoretical results. Recall that this method has two degrees of freedom (only) for all edges at the interface $\partial\Omega_S$ where $\sigma$ has a jump.
The $\N_0$-method, on the other hand, uses algebraic reduction also at the interface, thus violating assumption (A5). This indeed results in a reduction of the convergence rate, as could be expected from our analysis. 
Condition (A5) therefore seems necessary to obtain the optimal convergence rates in the presence of discontinuous conductivities.
In further tests, we also verified that 
assumption (A5$^*$) is necessary if the right-hand sides are chosen according to \eqref{eq:rhs2}.

\subsection*{Additional consistency errors}
To better understand the convergence breakdown when violating the consistency conditions (A5) and (A5$^*$), we illustrate in Figure~\ref{fig:err} the local contributions of the errors $\tnorm E_h^\N-E_h^\NC\tnorm$ for $h=2^{-4}$ for the two choices \eqref{eq:rhs1} and \eqref{eq:rhs2} of the right-hand sides.
\begin{figure}[ht!]
\includegraphics[width=0.35\textwidth]{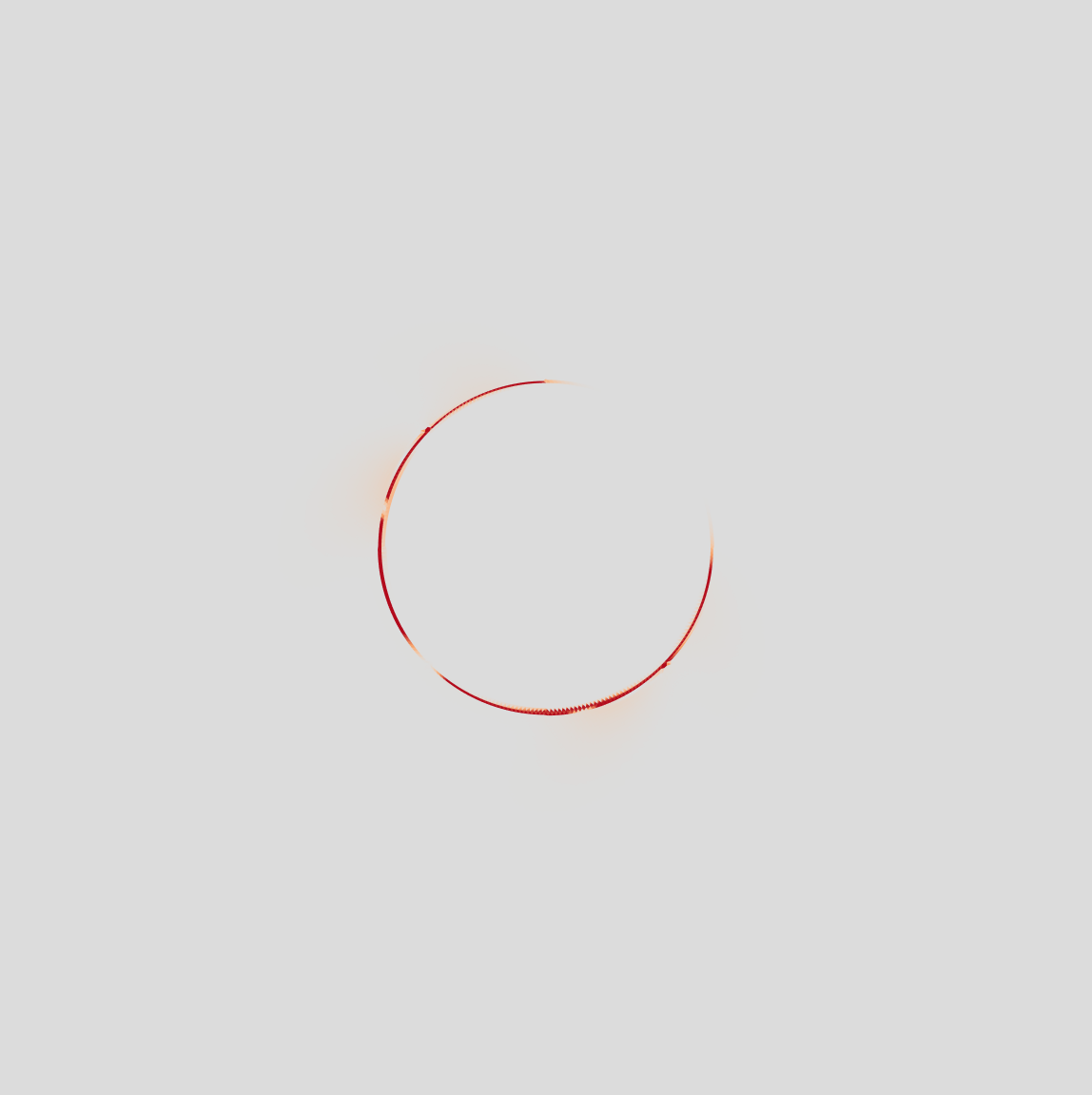}\quad
\includegraphics[width=0.35\textwidth]{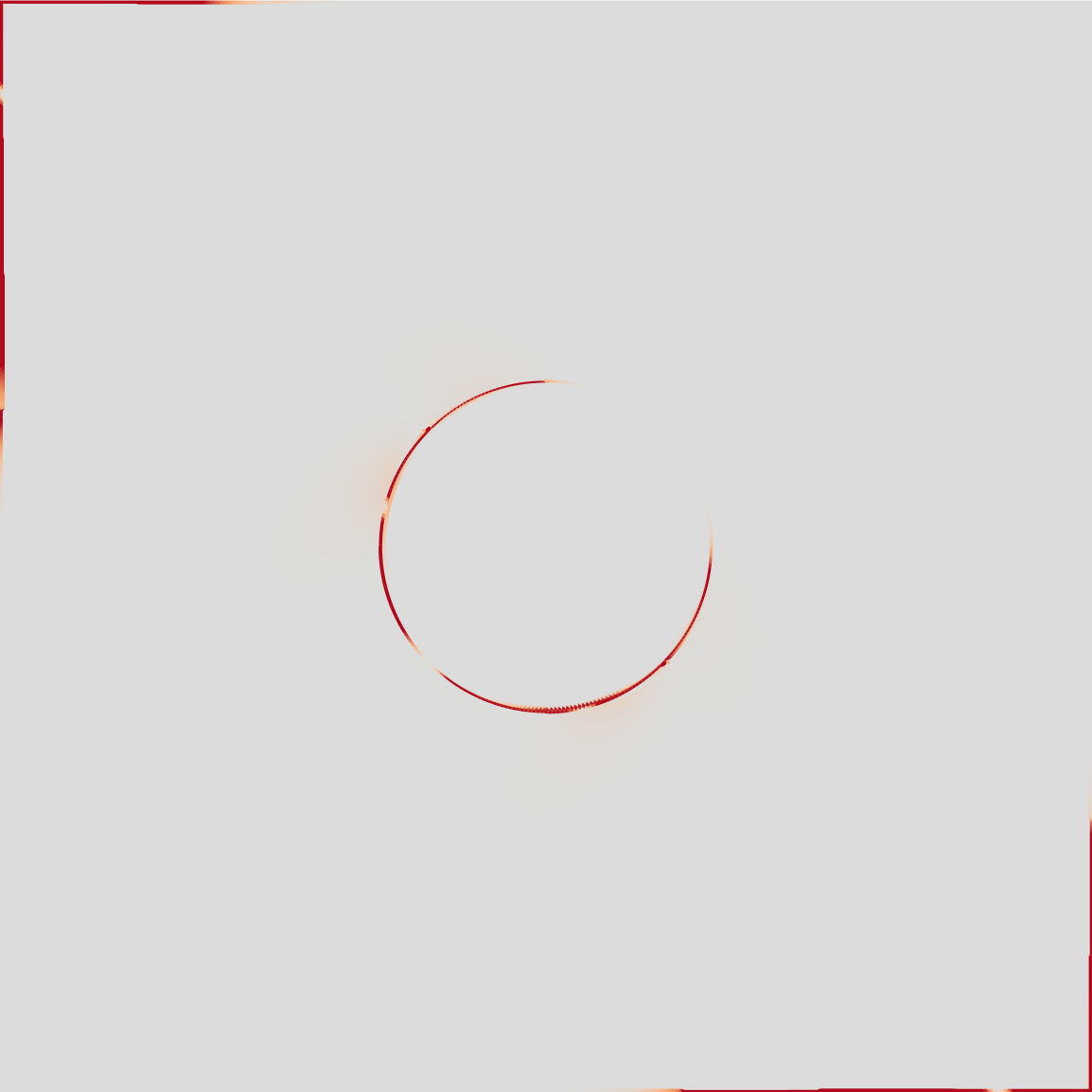}\quad
\captionsetup{width=.8\linewidth}
\caption{Local contributions of the error $\tnorm E_h^{\N}-E_h^\NC\tnorm $ for $h=2^{-4}$ with different implementations of the right-hand sides. The red color indicates which part of the mesh the error dominates. \label{fig:err}}
\end{figure}
As can be expected, the dominating error contributions come from the interface and the boundary, where the consistency conditions (A5) respectively (A5$^*$) are violated.

\subsection*{Inspection of the CFL condition}
As we have observed in assumption (A4), the CFL condition is somewhat non-standard, introducing an additional term depending on the projection $\wPi_h$.
To illustrate the effect of this additional term, we compute the largest time step $\tau=\tau_{\max}$ for which
\begin{align*}
\frac{\tau_{\max}^2}4 \ttv^\top\ttK_\nu \ttv - \frac12\ttv^\top\ttM_{\eps}\ttv + \gamma\frac{\tau_{\max}}2\Big|\ttv^\top\Big(\wh \ttM_\sigma-\ttM_\sigma\Big)\ttv\Big| \ge 0
\end{align*}
for all vectors $\ttv$, where $\gamma=0$ for the $\NC_1$ method and $\gamma=1$ for the $\N_0^+$ method.
By Taylor expansion, we note that $\tau_{\max}$ behaves approximately like
\begin{align}\label{eq:htau}
\tau_{\max} \sim \sqrt{\frac\varepsilon\nu} h + \gamma\frac\sigma{2\nu}h^2.
\end{align}
For $\gamma=0$ or $\sigma=0$, we obtain the classic linear dependence $\tau \le C h$.
In Table~\ref{tab:cfl}, we display the values of $C=\tau_{\max}/h$ for each of the methods for the model parameters used in our simulations.
\begin{table}[ht!]
\begin{tabular}{c||c|c} 
$h$ & $\NC_1$  & $\N_0^+$ \\
\hline
\hline
\rule{0pt}{2.1ex}
$2^{-3}$ & $0.391076$ & $0.210000$  \\
$2^{-4}$ & $0.381624$ & $0.317041$  \\
$2^{-5}$ & $0.377977$ & $0.377977$  \\  
$2^{-6}$ & $0.376846$ & $0.376846$  \\  
$2^{-7}$ & $0.376528$ & $0.376528$
\end{tabular}
\caption{CFL constants for both the $\NC_1$ and the $\N_0^+$ method on a sequence of uniformly refined meshes.\label{tab:cfl}}
\end{table}
First note that the CFL constant $C$ is uniformly bounded from above for both methods.
While for larger values of $h$, the CFL constant is somewhat more stringent for the $\N_0^+$ method, the constant $C$ behaves almost the same for small $h$, which is in perfect agreement with \eqref{eq:htau}.
The algebraic reduction hence does not have a severe effect on the maximal admissible time step.

\section{Discussion}\label{sec:discussion}

In this paper, we proposed and analyzed finite-element schemes which can be seen as a natural extension of the Yee scheme to unstructured grids and inhomogeneous lossy media.
While the number of degrees of freedom can be reduced to one for almost all edges, the necessity incorporation of two degrees of freedom at interfaces or boundaries with material jumps has been illustrated. 
A full convergence analysis of the schemes could be provided and optimal convergence could be proven under reasonable assumptions.
Alternative extensions of the Yee scheme to triangular and tetrahedral elements based on dual cell complexes have been proposed previously in \cite{CodecasaKapidaniSpecognaTrevisan18,CodecasaPoliti08,KapidaniCodecasaSchoberl21}.
While in some settings, the algebraic form of these methods is similar or even identical to ours, the constructions in these papers follow a rather geometric approach which complicates the error analysis. 
The results obtained in our paper may be useful to gain further insight also into these related methods.


\begin{thebibliography}{10}

\bibitem{BoffiBrezziFortin13}
D.~Boffi, F.~Brezzi, and M.~Fortin.
\newblock {\em Mixed finite element methods and applications}, volume~44 of
  {\em Springer Series in Computational Mathematics}.
\newblock Springer, Heidelberg, 2013.

\bibitem{BossavitKettunen99}
A.~Bossavit and L.~Kettunen.
\newblock Yee-like schemes on a tetrahedral mesh, with diagonal lumping.
\newblock {\em Int. J. Numer. Model.}, 12:129--142, 1999.

\bibitem{BrennerScott94}
S.~C. Brenner and L.~R. Scott.
\newblock {\em The mathematical theory of finite element methods}, volume~15 of
  {\em Texts in Applied Mathematics}.
\newblock Springer, New York, third edition, 2008.

\bibitem{CodecasaKapidaniSpecognaTrevisan18}
L.~Codecasa, B.~Kapidani, R.~Specogna, and F.~Trevisan.
\newblock Novel {FDTD} technique over tetrahedral grids for conductive media.
\newblock {\em IEEE Transactions on Antennas and Propagation}, 66:5387--5396,
  2018.

\bibitem{CodecasaPoliti08}
L.~Codecasa and M.~Politi.
\newblock Explicit, consistent, and conditionally stable extension of {FDTD} to
  tetrahedral grids by {FIT}.
\newblock {\em IEEE Transactions on Magnetics}, 44:1258--1261, 2008.

\bibitem{Cohen02}
G.~Cohen.
\newblock {\em Higher-Order Numerical Methods for Transient Wave Equations}.
\newblock Springer, Heidelberg, 2002.

\bibitem{CohenMonk95}
G.~Cohen and P.~Monk.
\newblock Efficient edge finite element schemes in computational
  electromagnetism.
\newblock In {\em Mathematical and numerical aspects of wave propagation},
  pages 250--259. SIAM, Philadelphia, PA, 1995.

\bibitem{CohenPernet16}
G.~Cohen and S.~Pernet.
\newblock {\em Finite Element and Discontinuous Galerkin Methods for Transient
  Wave Equations}.
\newblock Scientific Computation. Springer Netherlands, 2016.

\bibitem{EggerRadu20c_maxwellyee}
H.~Egger and B.~Radu.
\newblock A mass-lumped mixed finite element method for {M}axwell's equations.
\newblock In {\em Scientific computing in electrical engineering}, volume~32 of
  {\em Math. Ind.}, pages 15--24. Springer, Cham, 2020.

\bibitem{EggerRadu21a_maxwelltet}
H.~Egger and B.~Radu.
\newblock A second-order finite element method with mass lumping for
  {M}axwell's equations on tetrahedra.
\newblock {\em SIAM J. Numer. Anal.}, 59:864--885, 2021.

\bibitem{ElmkiesJoly97b}
A.~Elmkies and P.~Joly.
\newblock Éléments finis d'arête et condensation de masse pour les
  équations de {M}axwell: le cas {2D}.
\newblock {\em Comptes Rendus de l'Académie des Sciences - Series I -
  Mathematics}, 324:1287--1293, 1997.

\bibitem{ElmkiesJoly97}
A.~Elmkies and P.~Joly.
\newblock Éléments finis d'arête et condensation de masse pour les
  équations de {M}axwell: le cas de dimension 3.
\newblock {\em Comptes Rendus de l'Académie des Sciences - Series I -
  Mathematics}, 325:1217--1222, 1997.

\bibitem{ErnGuermond}
A.~Ern and J.-L. Guermond.
\newblock {\em Theory and practice of finite elements}, volume 159 of {\em
  Applied Mathematical Sciences}.
\newblock Springer-Verlag, New York, 2004.

\bibitem{FarwigRosteck2016}
R.~Farwig and V.~Rosteck.
\newblock Note on {F}riedrichs' inequality in {$N$}-star-shaped domains.
\newblock {\em J. Math. Anal. Appl.}, 435:1514--1524, 2016.

\bibitem{GeeversMulderVegt18}
S.~Geevers, W.~Mulder, and J.~van~der Vegt.
\newblock New higher-order mass-lumped tetrahedral elements for wave
  propagation modelling.
\newblock {\em SIAM Journal on Scientific Computing}, 40:A2830--A2857, 2018.

\bibitem{HesthavenWarburton08}
J.~S. Hesthaven and T.~Warburton.
\newblock {\em Nodal discontinuous {G}alerkin methods}, volume~54 of {\em Texts
  in Applied Mathematics}.
\newblock Springer, New York, 2008.
\newblock Algorithms, analysis, and applications.

\bibitem{Holland83}
R.~Holland.
\newblock Finite-difference solution of {M}axwell's equations in generalized
  nonorthogonal coordinates.
\newblock {\em IEEE Trans. on Nuclear Science}, NS-30:4589--4591, 1983.

\bibitem{Joly03}
P.~Joly.
\newblock Variational methods for time-dependent wave propagation problems.
\newblock In {\em Topics in Computational Wave Propagation}, volume~31 of {\em
  LNCSE}, pages 201--264. Springer, 2003.

\bibitem{KapidaniCodecasaSchoberl21}
B.~Kapidani, L.~Codecasa, and J.~Sch\"{o}berl.
\newblock An arbitrary-order {C}ell {M}ethod with block-diagonal mass-matrices
  for the time-dependent 2{D} {M}axwell equations.
\newblock {\em J. Comput. Phys.}, 433:Paper No. 110184, 20, 2021.

\bibitem{Lee92}
J.~F. Lee, R.~Pandalech, and R.~Mittra.
\newblock Modeling three-dimensional discontinuities in waveguides using
  nonorthogonal {FDTD} algorithm.
\newblock {\em IEEE Trans. on Microwave Theory and Techniques}, 40:346--352,
  1992.

\bibitem{Leis88}
R.~Leis.
\newblock {\em Initial Boundary Value Problems in Mathematical Physics}.
\newblock John Wiley, New York, 1988.

\bibitem{MakridakisMonk95}
C.~G. Makridakis and P.~Monk.
\newblock Time-discrete finite element schemes for {M}axwell's equations.
\newblock {\em RAIRO Mod\'{e}l. Math. Anal. Num\'{e}r.}, 29:171--197, 1995.

\bibitem{Monk92a}
P.~Monk.
\newblock Analysis of a finite element methods for {M}axwell's equations.
\newblock {\em SIAM J. Numer. Anal.}, 29:714--729, 1992.

\bibitem{Monk92c}
P.~Monk.
\newblock A comparison of three mixed methods for the time-dependent
  {M}axwell's equations.
\newblock {\em SIAM J. Sci. Statist. Comput.}, 13:1097--1122, 1992.

\bibitem{Monk93}
P.~Monk.
\newblock An analysis of {N}\'{e}d\'{e}lec's method for the spatial
  discretization of {M}axwell's equations.
\newblock {\em J. Comput. Appl. Math.}, 47:101--121, 1993.

\bibitem{Monk93a}
P.~Monk.
\newblock Finite element time domain methods for {M}axwell's equations.
\newblock In {\em Second {I}nternational {C}onference on {M}athematical and
  {N}umerical {A}spects of {W}ave {P}ropagation ({N}ewark, {DE}, 1993)}, pages
  380--389. SIAM, Philadelphia, PA, 1993.

\bibitem{Monk03}
P.~Monk.
\newblock {\em Finite element methods for {M}axwell's equations}.
\newblock Numerical Mathematics and Scientific Computation. Oxford University
  Press, New York, 2003.

\bibitem{MonkSuli94}
P.~B. Monk and E.~S{\"u}li.
\newblock A convergence analysis of {Yee}'s scheme on nonuniform grids.
\newblock {\em SIAM J. Numer. Anal.}, 31:393--412, 1994.

\bibitem{Nedelec86}
J.~C. N{\'e}d{\'e}lec.
\newblock A new family of mixed finite elements in $\mathbb{R}^3$.
\newblock {\em Numer. Math.}, 50:57--81, Jan 1986.

\bibitem{Nedelec80}
J.~C. Nédélec.
\newblock Mixed finite elements in $\mathbb{R}^3$.
\newblock {\em Numer. Math.}, 35:315--341, 1980.

\bibitem{Radu22}
B.~Radu.
\newblock {\em Finite element mass lumping for H(div) and H(curl)}.
\newblock PhD thesis, Technische Universit{\"a}t Darmstadt, Darmstadt, 2022.

\bibitem{SchuhmannWeiland98}
R.~Schuhmann and T.~Weiland.
\newblock {FDTD} on nonorthogonal grids with triangular fillings.
\newblock {\em IEEE Trans. Magn.}, pages 1470--1473, 1998.

\bibitem{SchuhmannWeiland98b}
R.~Schuhmann and T.~Weiland.
\newblock A stable interpolation technique for {FDTD} on nonorthogonal grids.
\newblock {\em Int. J. Numer. Model.}, 11:299--306, 1998.

\bibitem{Taflove80}
A.~{Taflove}.
\newblock Application of the finite-difference time-domain method to sinusoidal
  steady-state electromagnetic-penetration problems.
\newblock {\em IEEE Transactions on Electromagnetic Compatibility},
  EMC-22:191--202, Aug 1980.

\bibitem{Taflove05}
A.~Taflove and S.~C. Hagness.
\newblock {\em Computational electrodynamics: the finite-difference time-domain
  method}.
\newblock Artech House, Norwood, 3rd edition, 2005.

\bibitem{Weiland77}
T.~{Weiland}.
\newblock {A discretization model for the solution of {M}axwell's equations for
  six-component fields}.
\newblock {\em Archiv Elektronik und Uebertragungstechnik}, 31:116--120, Mar.
  1977.

\bibitem{Weiland03}
T.~Weiland.
\newblock Finite integration method and discrete electromagnetism.
\newblock In P.~Monk, C.~Carstensen, S.~Funken, W.~Hackbusch, and R.~H.~W.
  Hoppe, editors, {\em Computational Electromagnetics}, pages 183--198, Berlin,
  Heidelberg, 2003. Springer Berlin Heidelberg.

\bibitem{Yee66}
K.~{Yee}.
\newblock {Numerical solution of initial boundary value problems involving
  {M}axwell's equations in isotropic media}.
\newblock {\em IEEE Transactions on Antennas and Propagation}, 14:302--307, May
  1966.

\end{thebibliography}
\end{document}